\documentclass[12pt]{article}

\usepackage{amsmath,amsthm,amsfonts,txfonts,amssymb,amscd,epsf,epsfig,psfrag,enumerate,relsize}
\usepackage[mathscr]{eucal}
\usepackage{yhmath}
\usepackage{graphicx,epstopdf}
\usepackage{float}
\usepackage{wasysym}

\title{Lamination links in 3-manifolds }

\author{Ulrich Oertel}

\date{Version 2, December, 2018}

\newenvironment{tightenum}{
\begin{enumerate}[(1)]
  \setlength{\itemsep}{1pt}
  \setlength{\parskip}{0pt}
  \setlength{\parsep}{0pt}
}{\end{enumerate}}
\newenvironment{tightenumi}{
\begin{enumerate}[(i)]
  \setlength{\itemsep}{1pt}
  \setlength{\parskip}{0pt}
  \setlength{\parsep}{0pt}
}{\end{enumerate}}

\newtheorem{thm}{Theorem}[section] \newtheorem{lemma}[thm]{Lemma}

\newtheorem{corollary}[thm]{Corollary}

\newtheorem{proposition}[thm]{Proposition}
 \newtheorem*{claim*}{Claim}

 \theoremstyle{definition}
\newtheorem{defn}[thm]{Definition}
\newtheorem{defns}[thm]{Definitions} 
 \newtheorem{ex}[thm]{Example}

\newtheorem{remarks}[thm]{Remarks} 
\newtheorem{remark}[thm]{Remark}
\theoremstyle{remark}

\begin{document}

\maketitle

%\tableofcontents

\def\HDS{half-disk sum}

\def\PLKength{\text{Length}}

\def\Area{\text{Area}}
\def\Im{\text{Im}}
\def\im{\text{im}}
\def\cl{\text{cl}}
\def\rel{\text{ rel }}
\def\irred{irreducible}
\def\half{spinal pair }
\def\spinal{\half}
\def\spinals{\halfs}
\def\halfs{spinal pairs }
\def\reals{\mathbb R}
\def\rationals{\mathbb Q}
\def\complex{\mathbb C}
\def\naturals{\mathbb N}
\def\integers{\mathbb Z}
\def\id{\text{id}}
\def\Chi{\raise1.5pt \hbox{$\chi$}}
\def\cr{\tt\large}

\def\proj{P}
\def\hyp {\hbox {\rm {H \kern -2.8ex I}\kern 1.15ex}}

\def\Diff{\text{Diff}}

\def\weight#1#2#3{{#1}\raise2.5pt\hbox{$\centerdot$}\left({#2},{#3}\right)}
\def\intr{{\rm int}}
\def\inter{\ \raise4pt\hbox{$^\circ$}\kern -1.6ex}
\def\Cal{\cal}
\def\from{:}
\def\inverse{^{-1}}
\def\Max{{\rm Max}}
\def\Min{{\rm Min}}
\def\fr{{\rm fr}}
\def\embed{\hookrightarrow}
\def\Genus{{\rm Genus}}
\def\Z{Z}
\def\X{X}

\def\roster{\begin{enumerate}}
\def\endroster{\end{enumerate}}
\def\intersect{\cap}
\def\definition{\begin{defn}}
\def\enddefinition{\end{defn}}
\def\subhead{\subsection\{}
\def\theorem{thm}
\def\endsubhead{\}}
\def\head{\section\{}
\def\endhead{\}}
\def\example{\begin{ex}}
\def\endexample{\end{ex}}
\def\ves{\vs}
\def\mZ{{\mathbb Z}}
\def\M{M(\Phi)}
\def\bdry{\partial}
\def\hop{\vskip 0.15in}
\def\hip{\vskip0.1in}
\def\mathring{\inter}
\def\trip{\vskip 0.09in}
\def\PML{\mathscr{PML}}
\def\J{\mathscr{J}}
\def\A{\mathscr{A}}
\def\G{\mathscr{G}}
\def\F{\mathscr{F}}
\def\H{\mathscr{H}}
\def\C{\mathscr{C}}
\def\S{\mathscr{S}}
\def\S{\mathscr{S}}
\def\V{\mathscr{V}}
\def\CT{\mathscr{CT}}
\def\WS{\mathscr{WS}}
\def\PS{\mathscr{PS}}
\def\I{\mathscr{I}}
\def\PI{\mathscr{PI}}
\def\T{\mathscr{T}}
\def\PT{\mathscr{PT}}
\def\WT{\mathscr{WT}}
\def\PWT{\mathscr{PWT}}
\def\E{\mathscr{E}}
\def\K{\mathscr{K}}
\def\L{\mathscr{L}}
\def\PC{\mathscr{PC}}
\def\PWC{\mathscr{PWC}}
\def\W{\mathscr{W}}
\def\WC{\mathscr{WC}}
\def\PWC{\mathscr{PWC}}
\def\RW{\mathscr{RW}}
\def\RC{\mathscr{RC}}
\def\PLK{\mathscr{PLK}}
\def\CB{\mathscr{CB}}
\def\B{\mathscr{B}}
\def\PCB{\mathscr{PCB}}
\def\PT{\mathscr{PT}}
\def\W{\mathscr{W}}
\def\MC{\mathscr{MC}}
\def\CIB{\mathscr{CIB}}

\def\PMF{\mathscr{PMF}}
\def\OO{\mathscr{O}}
\def\OM{\mathscr{OM}}
\def\POM{\mathscr{POM}}
\def\BOM{\mathscr{BOM}}
\def\IM{\mathscr{IM}}
\def\PIM{\mathscr{PIM}}
\def\BIM{\mathscr{BIM}}
\def\PBIM{\mathscr{PBIM}}
\def\OMR{\mathscr{OMR}}
\def\POMR{\mathscr{POMR}}
\def\Preals{\mathscr{P}\mathbb R}
\def\PAA{\P\hskip -1mm\left[\reals^\A\times\reals^\A\right]}
\def\P{\mathscr{P}}
\def\suchthat{|}
\newcommand{\bigins}{\mathop{\mathlarger{\mathlarger\circledwedge}}}
\newcommand\invlimit{\varprojlim}
\newcommand\congruent{\equiv}
\newcommand\modulo[1]{\pmod{#1}}
\def\ML{\mathscr{ML}}
\def\Stack{\mathscr{T}}
\def\M{\mathscr{M}}
\def\A{\mathscr{A}}
\def\R{\mathscr{R}}
\def\union{\cup}
\def\atlas{\mathscr{A}}
\def\Int{\text{Int}}
\def\frontier{\text{Fr}}
\def\composed{\circ}

\def\abs{\odot}
\def\DS{\breve S}
\def\DL{\breve L}
\def\DBL{\breve{\bar L}}
\def\II{[0,\infty]}
\def\equiv{\hskip -3pt \sim}
\def\Chim{\Chi_-}

\def\split{\prec}
\def\pinch{\succ}
\def\OB{\mathbb O}
\def\FB{\mathbb F}
\def\SB{\mathbb S}

\def\TB{\mathbb T}
\def\OBB{{\mathbb S}\kern -6pt\raisebox{1.3pt}{--} \kern 2pt}
\def\Infty{\hbox{$\infty$\kern -8.1pt\raisebox{0.2pt}{--}\kern 1pt}}
\def\ens{\bar\circledwedge}
\def\ins{\circledwedge}
\def\rins{\circledvee}
\def\rens{\bar\circledvee}
\def\tmax{$\tau$-maximal-$\Chi$ branched surface}
\def\maxc{maximal-$\Chi$}
\def\isom{\cong}
\def\bov{{\bf v}}
\def\boh{{\bf h}}
\def\boa{{\bf a}}
\def\boc{{\bf c}}
\def\bod{{\bf d}}
\def\bob{{\bf b}}
\def\bow{{\bf w}}
\def\bou{{\bf u}}
\def\boy{{\bf y}}
\def\bor{{\bf r}}
\def\bot{{\bf t}}
\def\boq{{\bf q}}
\def\boz{{\bf z}}
\def\bos{{\bf s}}

\vskip 0.3in
\begin{abstract}  We introduce and define ``oriented framed measured lamination links" in a 3-manifold $M$.  These generalize oriented framed links in 3-manifolds, and are confined to 2-dimensional improperly embedded subsurfaces of the 3-manifold.  Just as some framed links bound Seifert surfaces, so also some framed lamination links bound 2-dimensional measured and oriented  ``Seifert laminations."   We show that any lamination link which bounds a 2-dimensional Seifert lamination, bounds a ``taut" Seifert lamination, i.e. one of maximum Euler characteristic subject to the condition that the Seifert lamination is carried by an aspherical branched surface, and satisfying further conditions.   The maximum Euler characteristic function is continuous on certain parametrized families of lamination links carried by a train track neighborhood.   Taut Seifert laminations generalize minimal genus Seifert surfaces.
\end{abstract}

\section{Introduction.}  Is there a space of oriented framed links in $S^3$?  This would be a space which includes points representing all classical oriented links in $S^3$.   The space would be non-compact and would include additional points beyond points representing classical framed links.  In this paper, we do not answer the question whether a space exists, but we prepare the ground.  Our experience of measured lamination spaces in surfaces will serve as a guide, though the end result is very different.  ``Links" in a surface are curve systems, and we know that the curve complex for a surface can be enlarged to a projective measured lamination space.  Thus it is reasonable to try to define knotted and or linked measured laminations embedded in a 3-manifold, which we will call ``framed oriented measured lamination links," or often just ``lamination links." Perhaps the most distinctive property of oriented knots and links is that they bound Seifert surfaces, so we shall pay special attention to framed measured lamination links which bound 2-dimensional oriented measured ``Seifert laminations."   It is best to describe the kind of object we seek using an example:

\begin{example} \label{IntroEx} Figure \ref{SeifertLam} shows an oriented branched surface $B$ in $S^3$.  We can assign weights to the sectors of the branched surface satisfying the branch equations at branch curves.  Suppose the weights are $x,y,z$ as shown, and the weight vector is $\bov=(x,y,z)$, satisfying $2z+2y=3x$.  If we assign integer weights, $B(\bov)$ is an oriented surface, determined by the weights, whose boundary is an oriented link (with a framing).  Even if the weights are rationally related, up to projective equivalence of weight vectors, $\bdry B(\bov)$ represents an oriented link.  But there are, of course, weight vectors satisfying the branch equations whose entries are not rationally related.  Such a  weight vector represents a ``Seifert lamination" whose boundary is an oriented measured lamination link which is not a classical link.  

To picture the lamination $B(\bov)$, it is best to replace $B$ by a certain fibered neighborhood we call $V(B)$ as shown in Figure \ref{BranchedNeighborhood}.  There is a projection $\pi:V(B)\to B$, and $\pi\inverse(\bdry B)$ is a 2-dimensional train track neighborhood which we will call $V(\tau)$, where $\tau=\bdry B\embed S^3$ is a train track.   (Throughout this paper, the symbol $\embed$ means ``embedded in.")  Observe that $V(\tau)$ is (transversely) oriented, since it lies in $\bdry V(B)$, which has an outward orientation.  Observe also that $\tau$ has an orientation induced by the orientation of $B$.   We say $V(\tau)$ with its orientation is a ``framing" of the oriented train track $\tau$ in $S^3$, which is determined here by the embedding of $B$, though in general a framing can be chosen at random.  The weights on sectors of $B$ give weights on the train track $\tau$ which also satisfy branch or switch equations, and we can imagine the knotted lamination as $V(\tau)$ with the width of different parts of $V(\tau)$ given by the weights of an invariant weight vector $\bow$ for $\tau$, see Figure \ref{LamLink}.  The framed train track $\tau$ together with the invariant weight vector $\bow$ for $\tau$ determines a ``prelamination" $V_\bow(\tau)$.   This is a singular foliation of $V(\tau)$, with leaves transverse to fibers of $V(\tau)$ and with a transverse measure with the property that the measure of a fiber corresponding to a point in the interior of a segment of $\tau$ equals the weight $w_k$ assigned by $\bow$ to the segment.  If $s$ is a segment of $\tau$, $\pi\inverse(\intr(s))$ is foliated as a product and $V_\bow(\tau)$ has a singular foliation (with singularities on $\bdry V(\tau)$) as shown in Figure \ref{BranchedNeighborhood}.  In much the same way the invariant weight vector $\bov$ for $B$ determines a prelamination representing the Seifert lamination which we denote $B(\bov)$ or $V_\bov(B)$.  Here $V_\bov(B)$ is determined by $\bov$ and the embedding of $B$ in $M$.  

In this example the embedding of $V_\bow(\tau)$ embedded in $S^3$ represents a ``measured lamination link," and $V_\bov(B)$ embedded in $M$ represents a ``Seifert lamination" for the lamination link.  Observe that $\tau$ is oriented and $V(\tau)$ is oriented as a surface.  In this paper, we shall define lamination links requiring both of these orientations.  One could omit the requirement that $\tau$ be oriented for a more permissive definition of measured lamination links.

\begin{figure}[ht]
\centering
{\includegraphics{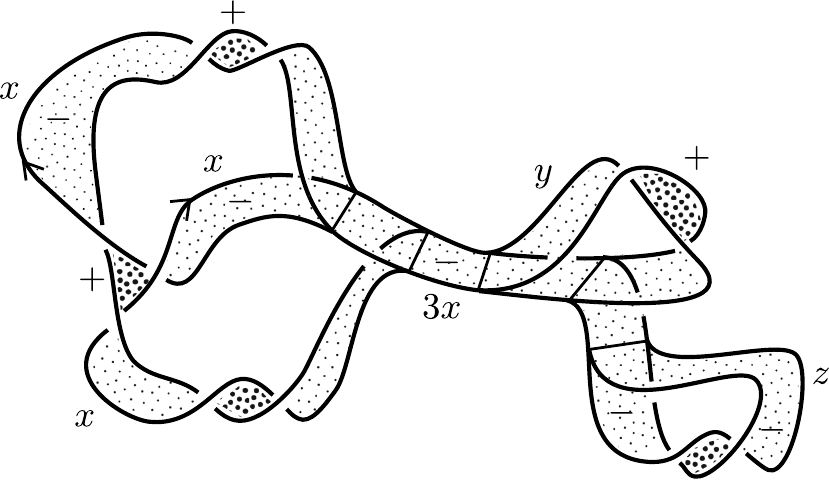}}
\caption{\footnotesize The branched surface $B$ carrying Seifert laminations for framed lamination links.}
\label{SeifertLam}
\end{figure}

\begin{figure}[ht]
\centering
{\includegraphics{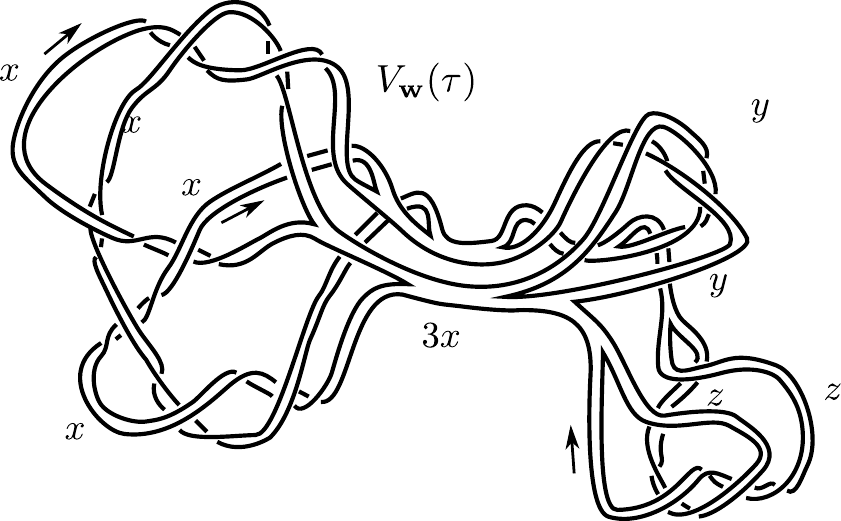}}
\caption{\footnotesize A lamination link.}
\label{LamLink}
\end{figure}

\begin{figure}[ht]
\centering
\scalebox{1}{\includegraphics{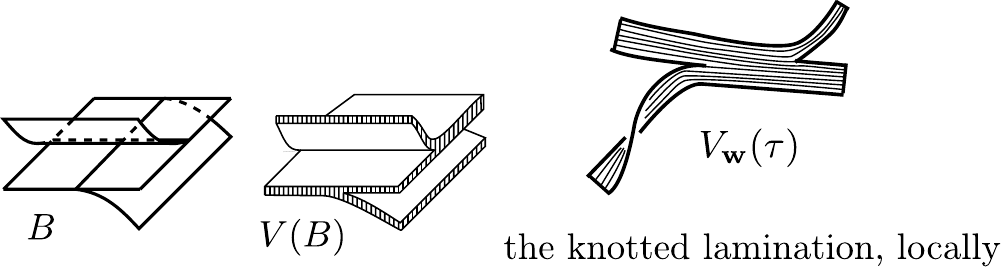}}
\caption{\footnotesize The branched surface $B$, its neighborhood $V(B)$, and the knotted lamination represented as $V_\bow(\tau)$ with weights giving ``widths."}
\label{BranchedNeighborhood}
\end{figure}

\end{example} 

\begin{remarks}  One might suppose that a more suitable definition of a lamination link would involve 3-dimensional regular neighborhoods of train tracks, with 2-dimensional disk fibers, and with the lamination transverse to the disk fibers.  Such a definition is much too permissive; it allows uncontrolled twisting and braiding in the fibered neighborhood.  Our lamination links are required to live in a 2-dimensional surface, namely a 2-dimensional train track neighborhood $V(\tau)$.    \end{remarks}

A fundamental invariant of an oriented knot is the minimal genus of a Seifert surface.  The genus of a Seifert surface can also be expressed in terms of the Euler characteristic of the Seifert surface, provided it contains no disks or spheres.  The Euler characteristic can be defined for measured 2-laminations (represented by prelaminations), see Section \ref{Preliminaries}.  Therefore, assuming a lamination link bounds some measured Seifert lamination, it is natural to ask whether there is a Seifert lamination with optimal Euler characteristic.   A Seifert lamination $B(\bov)$ is ``optimal" if it has maximal Euler characteristic subject to the condition that $B$ carry no sphere.    There is another manuscript, see \cite{UO:SeifertAlg}, which gives a kind of Seifert algorithm for constructing Seifert laminations when possible.   This addresses the question whether a given framed lamination link bounds a Seifert lamination.

For measured laminations carried by a branched surface $B$, the Euler characteristic is a linear function on the entries of a weight vector $\bov$ on $B$.  We write the Euler characteristic of a measured lamination $B(\bov)$ as $\Chi(B(\bov))$.   Alternatively, we can denote a measured 2-dimensional lamination $B(\bov)$ more abstractly as $(\Lambda,\nu)$, where $\nu$ is a transverse measure, assigning a measure to each interval intersecting $\Lambda$ transversely.  The abstract lamination $\Lambda$ is obtained by performing an ``infinite splitting" on the prelamination.

In order to give a formal definition of ``framed oriented measured lamination links," we must first define certain terms related to prelaminations.

\begin{defn}\label{SplittingDef} Suppose $V_\bow(\tau)$ is a prelamination, either abstract or embedded in a 3-manifold.  A {\it preleaf} of the prelamination is a leaf of the foliation of $V_\bow(\tau)\setminus \text{\{cusps\}}$ completed by including any cusps at one or both ends, see Figure \ref{Leaves}.   Let $C$ denote the set of cusps in $\bdry V(\tau)$.  A preleaf of the singular foliation on $V(\tau)$ determined by the weight vector $\bow$ is called a {\it separatrix} if its interior is contained in the interior of $V(\tau)$ and it has at least one end at a cusp point on $\bdry V(\tau)$.  The prelamination $V_\bow(\tau)$ is {\it irreducible} if it is fully carried and there are no compact separatrices, homeomorphic to a closed interval with each end at a cusp point.  A measured prelamination $V_\bow(\tau)$ is {\it fully carried} by $V(\tau)$ if all the entries of $\bow$ are strictly positive.

A {\it splitting} of the prelamination $V_\bow(\tau)$ is a prelamination $V_{\bow'}(\tau')$ obtained by cutting $V(\tau)$ on a compact submanifold (disjoint from $\bdry V(\tau)\setminus\{\text{cusps}\}$) of the union of preleaves.  We also say $\tau'$ is a {\it splitting} of $\tau$. 
\end{defn}

%A splitting of $V_\bow(\tau)$ is {\it regular} if it can be realized as a finite sequence of splittings of the following kind:   Split on a compact segment of a separatrix with exactly one end at a cusp.  

%A splitting is {\it good} if it can be realized as a finite sequence of splittings of the following kind:  Split on a compact segment of a separatrix with one or both ends at a cusp.  (Splitting on a closed preleaf is not allowed in a good splitting, nor is splitting on a compact arc in a preleaf with neither end of the arc at a cusp.)

%The inverse operation of a (good, regular) splitting is called a {\it (good, regular) pinching}.

%Clearly, regular pinching or splitting of a prelamination $V_\bow(\tau)$ preserves irreducibility of $V_\bow(\tau)$.

\begin{figure}[ht]
\centering
\scalebox{0.7}{\includegraphics{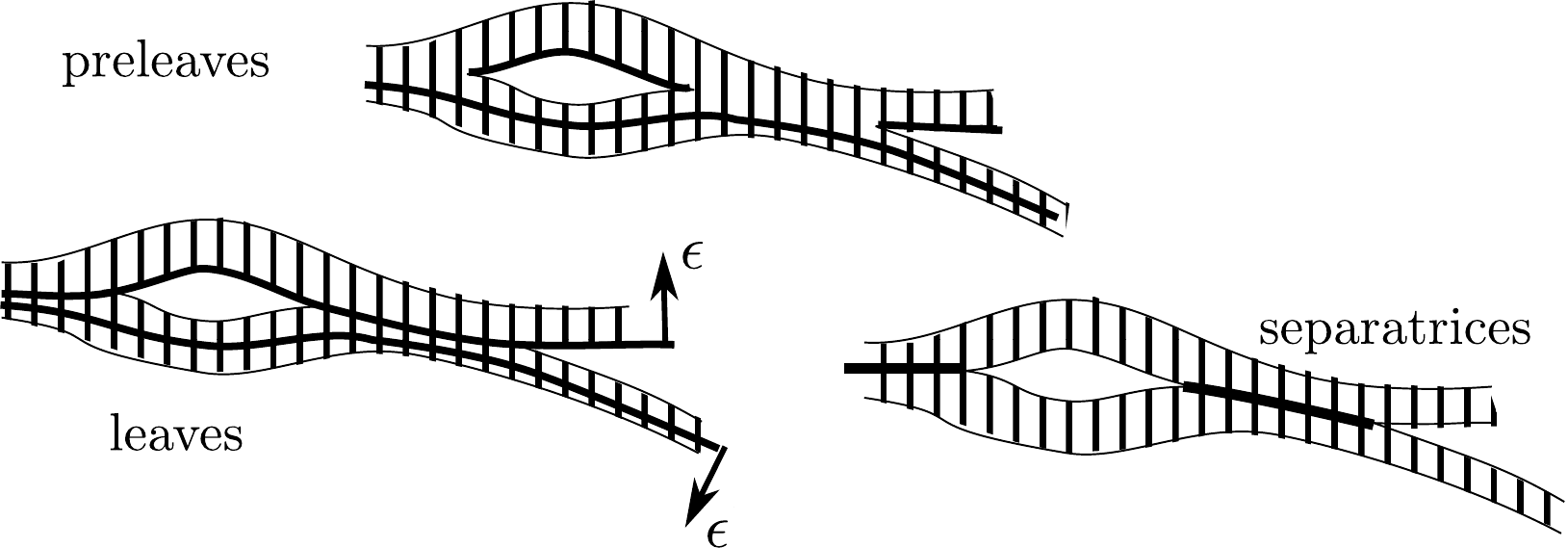}}
\caption{\footnotesize Preleaves and leaves of a prelamination."}
\label{Leaves}
\end{figure}

We will define the leaves of a prelamination following the approach in \cite{UO:MeasuredLaminations}.  For every point $x\in V_\bow(\tau)$ there are two possibilities for a transverse orientation $\epsilon$ at $x$.   We consider the set of pairs $(x,\epsilon)$ such that at points of $x\in \bdry V_\bow(\tau)\setminus C$ the transverse orientation points {\it into} the neighborhood $V_\bow(\tau)$.  

\begin{defn} \label{LeafDef} Two pairs $(x_0,\epsilon_0)$ and $(x_1,\epsilon_1)$ are in the same {\it leaf} of $V_\bow(\tau)$ if there is a path $(x(t),\epsilon(t))$, $0\le t\le 1$ such that $(x(i),\epsilon(i))=(x_i,\epsilon_i)$ for $i=0,1$.   

Two prelaminations  represented as $V_{\bow_1}(\tau_1)$ and $V_{\bow_2}(\tau_2)$ are {\it equivalent} if there is a splitting (in $M$) $V_{\bow_0}(\tau_0)$ of  $V_{\bow_1}(\tau_1)$ which, up to isotopy in $M$, is also a splitting of $V_{\bow_2}(\tau_2)$.   The equivalence classes are called {\it measured laminations}.  Finally $V_{\bow_1}(\tau_1)$  is {\it projectively equivalent} to $V_{\lambda \bow_1}(\tau_1)$  for any $0<\lambda<\infty$. 
\end{defn}

The definition of a leaf says roughly that locally in a foliation chart for $V_\bow(\tau)$ we have two leaves coinciding at a preleaf in the chart, except at $\bdry V_\bow(\tau)\setminus C$ where locally there is only one leaf corresponding to a preleaf, see Figure \ref{Leaves}.   If $B\embed M$ is a branched surface (possibly improperly) embedded in a 3-manifold $M$, the {\it leaves} of $B(\bov)$ are defined in exactly the same way, keeping in mind that $B(\bov)$ represents a certain $V_\bov(B)$ completely determined by $\bov$ and the embedding of $B$ in $M$.   Also, 2-dimensional measured laminations are defined as equivalence classes of prelaminations, as before.

There is another more common point of view on measured laminations, see Definition \ref{Neighborhood} and the discussion preceding it.    A measured lamination is a closed foliated subset $L$ of a manifold with an invariant transverse measure $\mu$.   If $(L,\mu)$ is a measured lamination in embedded in $V(\tau)$ transverse to fibers, it induces an invariant weight vector $\bow$ on $\tau$.   For all practical purposes, this is the same as an equivalence class of prelaminations $V_\bow(\tau)$.  The leaves are not quite the same; a leaf (or two leaves) of the representative prelamination may double-cover a leaf of the lamination $(L,\mu)$.  We will pass from one point of view to the other without comment.

There is an obvious notion of ``pinching in $M$."   Suppose $V_{\bow_i}(\tau_i)\embed M$ is a prelamination embedded in $M$, $i=0,1$.   Suppose there is a homotopy $V_{\bow_t}(\tau_t)$ of $V_{\bow_0}(\tau_0)$, $0\le t\le 1$, which is an isotopy for $0\le t<1$, and a pinching map $V_{\bow_0}(\tau_0)\to V_{\bow_1}(\tau_1)$ when $t=1$.  Then we say $V_{\bow_1}(\tau_1)$ is a {\it pinching in $M$} of $V_{\bow_0}(\tau_0)$, while $V_{\bow_0}(\tau_0)$ is a {\it splitting in $M$ of  $V_{\bow_1}(\tau_1)$}. 

\begin{defn} \label{KnottedDef1} 
%Suppose $V(\tau)\embed M$ is an oriented 2-dimensional fibered train track neighborhood embedded in an oriented 3-manifold $M$, where $\tau$ is an oriented train track.   A {\it complementary digon} for $V(\tau)$ is an embedded disk $D$ such that $D\cap V(\tau)=\bdry D$ and $\bdry  D$ contains exactly two cusp points of $V(\tau)$.    (Throughout this paper, the symbol $\embed$ means ``embedded in.")  We say $V(\tau)$ is {\it adigonal} if there are no complementary digons for $V(\tau)$, except complementary digons $D$ with the property that $D\cup V(\tau)$ contains a torus component.    A {\it complementary annulus} for $V(\tau)$ is an embedded annulus $A$ such that $A\cap V(\tau)=\bdry A$.   We say  $V(\tau)$ is {\it anannular} if there are no complementary annuli except complementary annuli with the property that $A\cup V(\tau)$ is a torus.

 %A framed measured oriented lamination link in an oriented 3-manifold $M$ is a 1- dimensional measured lamination which can be represented as a 2-dimensional  adigonal, anannular, oriented fibered train track neighborhood $V(\tau)\embed M$ for an oriented compact train track $\tau\embed M$ with an invariant weight vector $\bow$ on $\tau$  determining the prelamination, which we denote $V_\bow(\tau)$.   
  
 Suppose that we are given two prelaminations $V_{\bow_1}(\tau_1)$ and $V_{\bow_2}(\tau_2)$ both embedded in $M$, with both $\tau_i$ and $V(\tau_i)$ oriented.  They are {\it equivalent} if there is a splitting  in $M$, $V_{\bow_0}(\tau_0)$ of  $V_{\bow_1}(\tau_1)$ which, up to isotopy in $M$ is also a splitting in $M$, of $V_{\bow_2}(\tau_2)$.  A {\it framed measured oriented lamination link} is then the equivalence class of some $V_\bow(\tau)$.   We will sometimes refer to a  framed measured oriented lamination link simply as a {\it link}.   We say $V_\bow(\tau)$ {\it represents} the link.  

 If  $V_\bow(\tau)$ represents a lamination link and is irreducible and connected, the we say the link  is a {\it knot}.

A {\it trivial leaf}  is a closed leaf which bounds an embedded disk whose interior is disjoint from the lamination.   An oriented link in the usual sense will be called a {\it classical link}, and a classical link with a framing will be called a {\it classical framed link}.

A branched surface $B\embed M$ is {\it aspherical} if it does not carry any sphere.

Suppose $B\embed M$ is an oriented branched surface, and $\bdry B=\tau$ has the induced orientation.  An oriented 2-dimensional measured lamination $(\Lambda,\nu)$ represented as $V_\bov(B)$ is called a {\it Seifert lamination} for the link represented as  $V_\bow(\tau)$ if $V_\bow(\tau)$  is the oriented boundary of $V_\bov(B)$.  If the link  $V_\bow(\tau)$ has a Seifert lamination, we say the link {\it bounds}.  For any oriented $V_\bow(\tau)$ embedded in $M$, with $\tau$ also oriented, if $V_\bow(\tau)$ bounds,  we define $X( V_\bow(\tau))=\sup\{\Chi(V_\bov(B))\}$ where the supremum is over 2-dimensional laminations $V_\bov(B)$ carried by aspherical oriented branched surfaces $B\embed M$ and with $\bdry(V_\bov(B))=  V_{\bow'}(\tau')$ where   $V_{\bow'}(\tau')$ is equivalent to $V_\bow(\tau)$.  The lamination $V_\bov(B)$ is {\it maximal-$\Chi$} if it achieves the above supremum, and the supremum is finite.

  A {\it peripheral} (framed, measured, oriented) link in $M$ is a link $V_\bow(\tau)$ with $V(\tau)$ embedded in $\bdry M$ or isotopic into $\bdry M$.  \end{defn}

%Note that in case $V_\bow(\tau)$ has a complementary digon $D$ such that $D\cup V(\tau)$ contains a torus $T$, and $V_\bow(\tau)$ is a prelamination, then pinching $V_\bow(\tau)$ on the digon yields a measured foliation of $T$.   Further, if this torus is foliated by closed curves, then the prelamination $V_\bow(\tau)$ is reducible.  If $V_\bow(\tau)$ has a complementary annulus such that $A\cup V(\tau)$ contains a torus $T$, then pinching $V_\bow(\tau)$ on the annulus yields a  measured foliation of the torus $T$ by closed curves.

\vskip 0.2in

Peripheral links are easier to deal with, and they give a starting point for studying arbitrary links.  Observe that if $\tau\embed \bdry M$, there is an obvious choice of framing for $\tau$, namely we can take $V(\tau)$ to be a regular neighborhood in $\bdry M$ of $\tau$, and we are in a situation where the invariant weight vector on $\tau$ determines the peripheral link, so we can write $\tau(\bow)$ instead of $V_\bow(\tau)$ to denote the prelamination determined by the weights.   

The following theorem says that we can find a maximal-$\Chi$ Seifert lamination for any peripheral lamination link carried by $\tau\embed \bdry M$, provided the link bounds some Seifert lamination.   In fact, we can arrange that the Seifert lamination is carried by a branched surface with further nice properties, namely a {\it $\tau$-taut branched surface}, see Definition \ref{TautBranched}, which implies not only that the Seifert lamination is maximal-$\Chi$, but also that the leaves of the Seifert lamination are $\pi_1$-injective.  (In general, if $f:Y\embed X$ is an embedding of one topological space in another, we say $Y$ is {\it $\pi_1$-injective} if the induced map $f_*:\pi_1Y)\to \pi_1(X)$ is an injection.)  The Seifert laminations carried by $\tau$-taut branched surfaces are called {\it taut Seifert laminations}.

\begin{thm}  Given a peripheral link represented as $V_\bow(\tau)\embed \bdry M$ which bounds some Seifert lamination, there exists a taut Seifert lamination represented as $V_\bov(B)=B(\bov)$, $\bdry V_\bov(B)=V_\bow(\tau)$, where $B$ is $\tau$-taut.   In particular, $B(\bov)$ is maximal-$\Chi$ and every leaf of $B(\bov)$ is $\pi_1$-injective,
\end{thm}

\begin{defn}
The non-zero invariant weight vectors on any train track $\tau$ form a cone $\C(\tau)$.   Projectivizing by taking a quotient where $\bow\in \C(\tau)$ is equivalent to $\lambda\bow$, we obtain the {\it weight cell}, $\PC(\tau)$.  If $V(\tau)$ is any train track neighborhood of $\tau$, we also write  $\PC(\tau)=\PC(V(\tau))$.  Usually we use $\PC(\tau)$ to denote a particular subspace of $\C(\tau)$, namely $\PC(\tau)=\{\bow\in \C(\tau):\sum_iw_i=1\}$, where $w_i$ are the entries of $\bow$.
\end{defn}

Some peripheral links bound Seifert laminations, some do not.   We show in Section \ref{Peripheral} that the set of bounding links carried by $\tau\embed \bdry M$ corresponds to a convex subcone $\CB(\tau)\subset \C(\tau)$.  Projectivizing, we obtain $\PCB(\tau)\subset \PC(\tau)$.
  
  We will prove more detailed  results implying:

\begin{thm}\label{SeifertThm} Suppose $M$ is a compact, oriented 3-manifold.  Suppose $\tau$ is an embedded oriented train track in $\bdry M$.  Then

\noindent (a) the function $X:\CB(\tau)\to \reals$ defined by $X(\bow)=X(\tau(\bow))$ is continuous;

\noindent (b) if $\tau(\bow)$ represents a lamination which bounds  a 2-dimensional oriented lamination then there is a $\tau$-taut branched surface $B$ and an invariant weight vector $\bov$, with $\bdry B(\bov)=V_\bow( \tau)$.  In particular $B(\bov)$ achieves a finite supremum in the definition of $X(V_\bow(\tau))$ and the leaves of $B(\bov)$ are $\pi_1$-injective.
\end{thm}

The method of proof of the above theorem shows that there are some similarities with William Thurston's norm on homology, \cite{WPT:Norm}.   We will prove that the function $X$ is piecewise linear, and we will show that there are finitely many $\tau$-taut branched surfaces $B$ such that $\bdry B$ is a sub-train track of $\tau$ and for every $\tau(\bow)$, there is a branched surface $B$ in the finite collection and an invariant weight vector $\bov$ for $B$ such that $\bdry B(\bov)=\tau(\bow)$.

The following theorem clarifies the analogy between the function $X:\CB(\tau)\to \reals$ and Thurston's norm on homology.   

\begin{thm}\label{XThm} Suppose $M$ is a compact orientable 3-manifold.  Suppose $\tau\embed \bdry M$ is an embedded oriented train track in $\bdry M$.
\noindent The function $X:\CB(\tau)\to \reals$ is linear on rays, finite piecewise linear, and concave.
\hip
\noindent There exists a finite collection of $\tau$-taut branched surfaces such that for every peripheral link $\tau(\bow)$ which bounds a Seifert lamination there is a branched surface $B$ of the collection, and $\bov\in C(B)$ such that $B(\bov)$ is a taut Seifert lamination for $\tau(\bow)$.
\end{thm}

We will show that, in a sense, Theorem \ref{SeifertThm} also applies to non-peripheral links.   Given a lamination link represented as $V_\bow(\tau)$ which bounds a Seifert lamination, we will show that after possibly changing the representative $V_\bow(\tau)$, there is a submanifold $\hat M$ of $M$ such that  $V_\bow(\tau)$ is peripheral in $\hat M$, and we also show that a maximal-$\Chi$ Seifert lamination in $\hat M$ is maximal-$\Chi$ in $M$.   We give some more details.   Notice first that a lamination link $V_\bow(\tau)$ yields a measured singular foliation of a ``surface with inward cusps" or {\it scalloped surface} $S=V(\tau)$, where $S$ is the same as $V(\tau)$ but ignoring the fibering by intervals.  The ``cusps" of $S$ are the same as the cusps $C$ of $V(\tau)$, so the scalloped surface can be regarded as a pair $(S,C)$.  In our context, $S$ is also equipped with an orientation, and the curves of $\bdry S\setminus C$ are oriented.   Now a lamination link $V_\bow(\tau)\embed M$ can be viewed as a {\it measured oriented singular foliation $(\F,\mu)$} of a scalloped surface $S\embed M$ with no closed components, where the singularities of the foliation are the cusp points on $\bdry S$.    We are only interested in scalloped surfaces which support measured oriented foliations.  This means that $S$ has ``geometric Euler characteristic" 0.  Also, for example, if $S$ is an annulus with no cusps on its boundary, the two boundary components must be oriented such that one boundary component can be isotoped to the other preserving orientations.  

We have introduced scalloped surfaces not only because they are useful here, but also because we wish to use them in a later paper for defining a space of framed lamination links.   

\begin{defns}
Suppose $S\embed M$ is an oriented 2-dimensional scalloped surface embedded in an oriented 3-manifold $M$.   A {\it complementary digon} for $S$ is an embedded disk $D$ such that $D\cap S=\bdry D$ and $\bdry  D$ contains exactly two cusp points of $S$.   If $S\embed M$ is a scalloped surface, we say $S$ is {\it adigonal} if there are no complementary digons for $S$, except complementary digons $D$ with the property that $D\cup S$ contains a torus component containing $D$.    A {\it complementary annulus} for $S$ is an embedded annulus $A$ such that $A\cap S=\bdry A$ and the two curves of $\bdry S\cap A$ are isotopic in $A$ as oriented curves.   We say  $S$ is {\it anannular} if there are no complementary annuli except complementary annuli with the property that $A\cup S$ contains a torus containing $A$.   If $V(\tau)\embed M$ is a train track neighborhood, it is also a scalloped surface, so all of the above definitions (complementary digon, adigonal, complementary annulus, anannular) apply to $V(\tau)$ as well.  

{\it Preleaves and leaves} of a foliation $(\F,\mu)$ of a scalloped surface $S$ are defined just as for prelaminations $V_\bow(\tau)$.   A preleaf of the singular foliation $(\F,\mu)$ of $S$ is called a {\it separatrix} if its interior is contained int the interior of $S$ and has at has at least one end at a cusp point on $\bdry V(\tau)$.  Just as a prelamination $V_\bow(\tau)$ is irreducible if there is no separatrix preleaf with each end at a cusp, so also a measured foliation $(\F,\mu)$ of a scalloped surface $S$ is {\it irreducible} if it has no separatrix preleaf with each end at a cusp and in case $S$ contains a torus component, $\F$ has no closed leaf in the torus.
\end{defns}

The definition of anannular is designed to ensure that if $S\embed M$ is a single annulus supporting an oriented measured foliation, which must be a product foliation, then $S$ is anannular.  On the other hand, if $S$ consists of two annuli, and there is a complementary annulus $A$, then $S$ is  not anannular.   In this case we would like to combine the two components of $S$. 

We note that if a foliation $(\F,\mu)$ of a scalloped surface $S$ which is not a torus is {\it reducible}, i.e. if there is a separatrix with each end at a cusp of $S$, then cutting on the separatrix yields a foliation of a simpler scalloped surface (with fewer cusps).

\begin{proposition}  \label{AdigonalProp} 

(a) Given a scalloped surface with an irreducible, adigonal, and anannular measured oriented foliation, $(S,\F,\mu)$ there is a prelamination $V_\bow(\tau)$ whose underlying scalloped surface with measured oriented foliation is $(S,\F,\mu)$.

(b) If $V_\bou(\rho)\embed M$ represents a lamination link in $\intr(M)$, then $V_\bou(\rho)$ can be replaced by an irreducible, adigonal, anannular representative $V_\bow(\tau)$.   The underlying scalloped surface $S$ for $V_\bow(\tau)$   is unique up to isotopy in $M$ and is adigonal and anannular.   The measured foliation $(\F,\mu)=V_\bow(\tau)$ of $S$ is irreducible, adigonal, and anannular.   

\noindent   \end {proposition}

 There is a submanifold $\hat M$ of $M$, such that $S\embed \bdry \hat M$, and therefore the link represented by $V_\bow(\tau)$ is peripheral in $\hat M$.  The submanifold $\hat M$ is obtained as follows: We let $N$ be a regular neighborhood of $S$ in $\intr(M)$.  The neighborhood $N$ must be a product of the form $S\times I$, since $M$ and $N$ are orientable.  We let $\hat M$ be the closure of $M\setminus N$.    Then we isotope $S$ so $S\subset \bdry N=\bdry \hat M$ and so that the orientation induced on $S$ by the orientation on $\hat M$ is the same as the orientation on $S$.  
  
 \begin{defn} If $V_\bow(\tau)$ is irreducible, anannular, and adigonal, then we say $\hat M$ is the {\it track exterior} for the link $V_\bow(\tau)$.
 \end{defn}

Since the underlying scalloped surface $S$ for a lamination link represented as an irreducible, adigonal, anannular $V_\bow(\tau)$ is unique up to isotopy, the track exterior is well-defined.

\begin{thm} \label{Extend}  Suppose $M$ is a compact, oriented 3-manifold and $V_\bow(\tau)$ is an irreducible, adigonal, anannular representative of a lamination link embedded in $M$.   Then there is a maximal-$\Chi$ Seifert lamination for $V_\bow(\tau)$.   Further, there is a maximal-$\Chi$ Seifert lamination $B(\bov)$ which is also a taut Seifert lamination (carried by a $\tau$-taut branched surface $B$) when  $V_\bow(\tau)$ is viewed as a peripheral link in the track exterior $\hat M$.  In particular, leaves of $B(\bov)$ are $\pi_1$-injective into $\hat M$.  
\end{thm}

We can also make a statement about the continuity of the function $X$ for non-peripheral links.   

\begin{defn}  Suppose $V(\tau)\embed M$ is a framed adigonal, anannular train track.    We let $\CIB(V(\tau))$ denote the subspace of $\CB(V(\tau))$ of invariant weight vectors $\bow$, with all entries positive, on $\tau$ such that $V_\bow(\tau)$ is irreducible and bounds a Seifert lamination.  $$\CIB(V(\tau))=\{\bow\in \CB((\tau)):V_\bow(\tau) \text{ is irreducible and every entry } w_i>0\}$$
\end{defn}

\begin{thm}\label{Continuity}  Suppose $V(\tau)\embed M$ is a framed adigonal, anannular train track in an oriented compact 3-manifold.   Then the function $X:\CIB(V(\tau))\to \reals$  defined by $X(\bow)=X(V_\bow(\tau))$ is continuous. 
\end{thm}

We hope to use this last theorem to describe continuity properties of $X$ on a space of links.

The following theorem confirms a natural presumption about trivial leaves in a lamination link:

 \begin{thm}\label{Trivial}  If a lamination link $V_\bow(\tau)$ (peripheral or not) in a compact 3-manifold bounds a Seifert lamination and contains trivial leaves, then a taut  Seifert lamination $V_\bov(B)$ for $V_\bow(\tau)$ has  the property that  every trivial leaf of $V_\bow(\tau)$ bounds a disk in a leaf of the Seifert lamination $V_\bov(B)$.
\end{thm}

\section{Preliminaries}\label{Preliminaries}

\begin{defns}  Suppose a surface $V$ can be subdivided into rectangles or {\it charts} of the form $I\times J$, where $I=[0,1]$ and $J=[0,w]$ foliated by leaves $I\times \{u\}$ and by fibers $\{t\}\times J$. We require that $\bigcup I\times \bdry J\subset \bdry V$, where the union is over charts.   Another chart (or the same one) intersects $I\times J$ in sub-intervals of fibers in $\bdry I\times J$.   Then $V$ with the singular foliation obtained as the union of foliated charts is a {\it dimension 1 prelamination}.   If adjacent charts intersect such that the metrics on intersecting fibers coincide on the intersection, the prelamination is {\it transversely measured}.

  The chart foliations yield a foliation of $V\setminus C$, where $C$ is the set of {\it cusp points} on $\bdry V$, which are points that lie on the interior of a fiber of one chart and also at the ends of fibers of two adjacent charts.  A {\it preleaf} of the prelamination is a leaf of $V\setminus C$ completed by any cusp points at the ends of the leaf.  See Figures \ref{Leaves} and \ref{PreLam}.   A {\it splitting} of a prelamination, is another prelamination obtained by cutting on a finite union of compact submanifolds of preleaves, provided these submanifolds intersect $\bdry V(\tau)$ only at cusps.   {\it Pinching} is the opposite of splitting.  

We can make similar definitions in other dimensions.  A {\it  dimension $m$ prelamination} is constructed using charts of the form $I\times J$ where $I$ is a disk of dimension $m$ and $J$ is an interval as before; the charts intersect only at their vertical  boundaries  $\bdry I\times J$ and cover an $(m+1)$- manifold $V$ with an inward cusp $(m-1)$-manifold $C$ in $\bdry V$.  (One can allow $J$ of higher dimension but we will not use such prelaminations.)

\begin{figure}[ht]
\centering
\scalebox{0.8}{\includegraphics{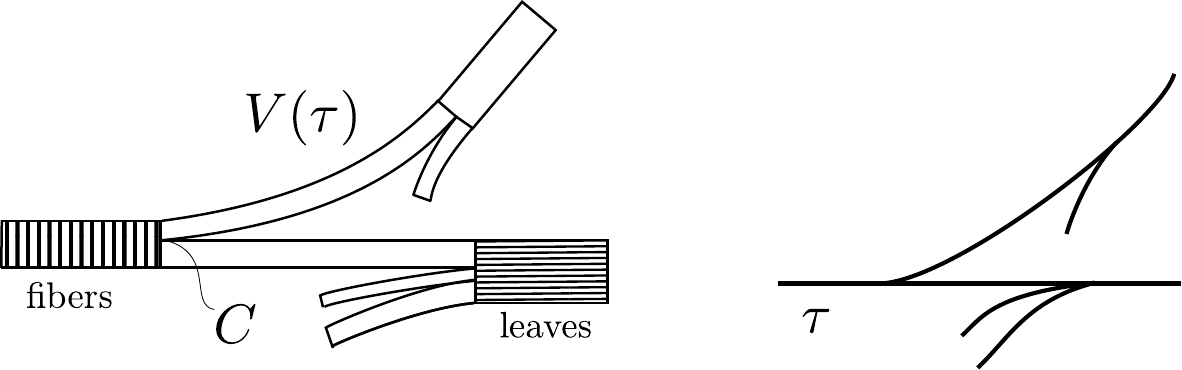}}
\caption{\footnotesize Chart structure for a prelamination.}
\label{PreLam}
\end{figure}

When a prelamination (of dimension $m$) is embedded in a manifold $M$ of dimension $m+1$, then we often call the embedded submanifold $V$ with its foliation by fibers a {\it fibered neighborhood of a branched manifold} of dimension $m$.  Even in the absence of a foliation transverse to the fibers, we refer to the union of charts foliated by fibers as a {\it fibered neighborhood}.  A {\it branched manifold} is a quotient obtained from a fibered neighborhood by identifying all points on any fiber, where the quotient map is often thought of as a projection.   For example, if $m=1$, the quotient is called a {\it train track} $\tau$, the fibered neighborhood is denoted $V(\tau)$ and the quotient map is $\pi:V(\tau)\to \tau$.  For $m=2$, a fibered neighborhood $V(B)$ projects to a {\it branched surface} $B$.  In any dimension the quotient $B$ is given a smooth structure such that a smooth $m$-disk locally embedded transverse to fibers of projects to a smooth disk in the quotient branched manifold.  In any dimension, we let $C$ be the {\it cusp locus} in $\bdry V$.   Then $\pi(C)$ is called the {\it branch locus}  ({\it switch points} for train tracks) and the completions of the components of the complement of the branch locus are called {\it sectors}.  For a train track $\tau$ every sector is either a closed curve component of $\tau$ or a sector homeomorphic to an interval, which we call a {\it segment}, .    We will use the notation $V(\tau)$ to denote a 2-dimensional fibered neighborhood even when $V(\tau)$ is embedded in a 3-manifold.  

A prelamination constructed as above is {\it fully carried} by the corresponding branched manifold.  

A codimension-1 {\it lamination} $L\embed M$ in an $m+1$-manifold $M$ is a closed foliated subset of $M$.

If a 1-dimensional prelamination is transversely measured, then there is a chart corresponding to a segment $s_i$, which has the form $I\times [0,w_i]$, and is assigned a weight $w_i>0$.   For a closed curve sector $s_i$, there is a chart of the form $S^1\times [0,w_i]$, and we again assign a weight $w_i$ to the sector.  In this case we could equally well use a chart of the form $I\times [0,w_i]$.  The weights yield an ``invariant weight vector", which is characterized by the following property:  If $s_1$, $s_2$, and $s_3$ have a common switch point $P$, and both $s_1\cup s_3$ and $s_2\cup s_3$ are locally smooth at $P$, then $w_1+w_2=w_3$.   This last equation is called a {\it switch equation}.   An {\it invariant weight vector} is a weight vector $\bow$ which assigns $w_i\ge 0$ to each segment $s_i$ of $\tau$ and such that the entries of the weight vector satisfy all the switch equations for $\tau$. \end{defns}

Clearly if $\tau$ is embedded in a surface, an invariant weight vector for $\tau$ with positive entries determines a measured prelamination $\tau(\bow)$ or $V_\bow(\tau)$; if the invariant weight vector does not have positive entries, it determines a prelamination fully carried by a sub- train track of $\tau$.  Also, given $V(\tau)$ a fibered neighborhood, the invariant weight vector determines a prelamination $V_\bow(\tau)$.

There are similar definitions of invariant weight vectors for branched surfaces and branched manifolds.   If $B$ is a branched surface embedded in a 3-manifold $M$, then an invariant weight vector $\bov$ determines a lamination $B(\bov)$.    If $B$ happens to be a surface $F$, the components are sectors.  An invariant weight vector with equal weights $\delta$ on all components determines a {\it weighted surface} $\delta F$.  

There is another point of view on laminations and fibered neighborhoods which is sometimes preferable.   A lamination can be constructed as an inverse limit of prelaminations obtained by performing ``infinite splitting" along leaves intersecting the cusp locus $C$ of a prelamination.  This can be done such that the leaves of the lamination are smooth with respect to a smooth structure on $M$.   A transverse measure on a prelamination then gives a transverse measure on the corresponding lamination.   When dealing with laminations rather than prelaminations, one uses a different kind of fibered neighborhood, as in the following definitions.

\begin{figure}[ht]
\centering
\scalebox{0.8}{\includegraphics{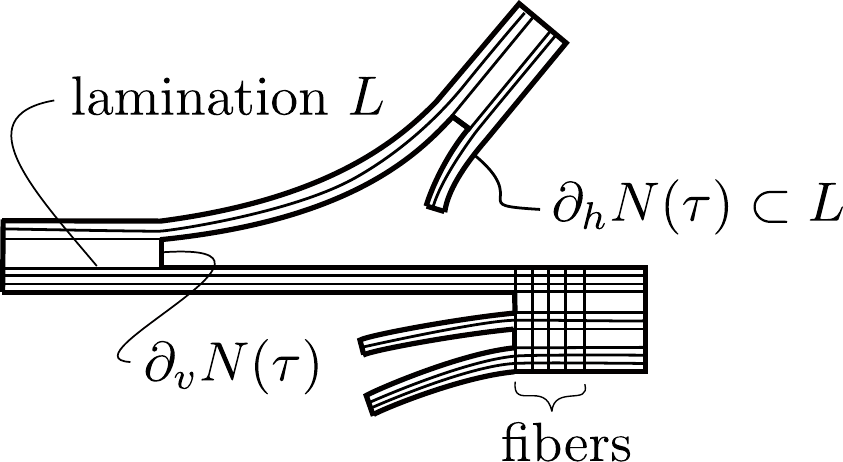}}
\caption{\footnotesize Fibered neighborhood $N(\tau)$, with vertical boundary, of a branched manifold.}
\label{Lam}
\end{figure}

\begin{defns} \label{Neighborhood} If $B\embed M$ is a codimension-1 branched manifold of dimension $m$ embedded in a manifold $M$, there is a corresponding fibered neighborhood $N(B)\embed M$.  Once again, $N(B)$ is ``vertically" foliated by {\it fibers} and there is a quotient map $\pi:N(B)\to B$ which identifies all the points on each fiber and yields the branched manifold $B$.  In Figure \ref{Lam} we show a typical fibered neighborhood when $B=\tau$ is a train track and $M$ is a surface.  Instead of a cusp locus there is a {\it vertical boundary} $\bdry_vN(B)$ contained in $\bdry N(B)$, which is an $I$-bundle over an $(m-1)$-manifold.   For example, if $B$ is a branched surface in a 3-manifold, the vertical boundary is a union of annuli, fibered by intervals.   The {\it horizontal boundary} $\bdry_hN(B)$ is the closure of $\bdry N(B)\setminus\left(\bdry_vN(B)\cup \bdry M\right)$.

A lamination $L\embed M$ is {\it carried} by $B$ if, after isotopy, it can be embedded in $N(B)$ transverse to fibers.   It is {\it fully carried} by $B$ if it can be isotoped so it is contained in $N(B)$ and transverse to fibers and intersecting every fiber.    If it is fully carried by $B$, $L$ can be modified slightly (by replacing any isolated leaf by the boundary of an immersed regular neighborhood of the leaf) so that $\bdry_hN(B)\subset L$.   When $\bdry_hN(B)\subset L$, the completion of $N(B)\setminus L$ has the structure of an $I$-bundle called the {\it interstitial bundle}.    If this bundle has the form $p:J\to W$, where $W$ is a surface with boundary, then we use $\bdry_vJ$ to denote $p\inverse(\bdry W)$.

If $B\embed M$ is a branched surface, a {\it disk of contact} for $B$ is a disk $D\embed N(B)$ transverse to fibers with $\bdry D\embed \intr(\bdry_vN(B))$.  Sometimes if $F$ is a surface fully carried by $B$, and we have  $\bdry_hN(B)\subset F$, a disk of contact $D$ can be isotoped vertically so that $D\subset F$.  In this case, we say $D$ is a {\it disk of contact in $F$} although $\bdry D\nsubset \intr(\bdry_vN(B))$.
\end{defns}

The notion of splitting a branched manifold is most easily described using the fibered neighborhood $N(B)$.

\begin{defn}  If $B\embed M$ and $B'\embed M$ are codimension-1 branched manifolds, then $B'$ is a {\it splitting} of $B$ if $N(B)=N(B')\cup J$, where $J$ is (the total space of) an $I$-bundle $p:J\to W$ over a surface $W$ and:

\begin{tightenum}
\item $\bdry_hJ\subset \bdry_hN(B')$,
\item $J\cap N(B')\subset \bdry J$, $J\cap N(B')\subset\bdry N(B')$,
\item $p\inverse(\bdry W)\subset \bdry J$ intersects $\bdry N(B')$ in finitely many components, each contained in $\bdry_vN(B')$, with each fiber in $\bdry J$ contained in a fiber of $N(B')$.
\end{tightenum}
\end{defn}

\begin{defn}  If $B\embed M$ is a branched surface embedded in a 3-manifold, it has {\it generic branch locus} if the projection $V(B)\to B$ maps the cusp locus $C$ to $B$ such that self-intersections of $C$ are transverse and in general position.  Then the sectors of $B$ are {\it surfaces with corners}, see Figure \ref{BranchedNeighborhood}.  If $Z$ is such a sector, we assign a {\it geometric Euler characteristic} to $Z$, $\Chi_g(Z)=\Chi(Z)-(1/4)k$, where $k$ is the number of corners of $Z$.   

If $\bov$ is an invariant weight vector for a branched surface $B$ and $v_i$ is the weight on the sector $Z_i$, then we define $\Chi(B(\bov))=\sum_i v_i\Chi_g(Z_i)$.

\end{defn}

There is some work needed to make sense of the above definitions.   For example, one must show that the Euler characteristic of a measured lamination is well-defined, not depending on the choice of $B$ and $\bov$ used to represent the lamination.

\begin{defn}  Suppose $(B,\bdry B)\embed (M,\bdry M)$.   Suppose $\bov$ is an invariant weight vector on a branched surface $B$.   Then we let $\bdry \bov$ denote the invariant weight vector induced on the train track $\bdry B$, obtained by restricting $\bov$ to $\bdry B$. 
\end{defn}

Note that each entry of $\bdry\bov$ is an entry of $\bov$, but a sector $Z$ of $B$ can intersect $\bdry B$ in several sectors of the train track, so an entry of $\bov$ can appear more than once in $\bdry \bov$.

\begin{defns}  
%An {\it essential train track} in a surface $S$ is an embedded train track with no complementary 0-gons or monogons. 
%If $\tau$ has $p$ segments, a weight vector $\bow$ assigns a {\it weight} $w_i$ to each segment $s_i$.  The weight vector is {\it invariant} if it satisfies all the {\it switch equations}:  For each switch point in $\tau$ there are three segments attached to the switch point, say $s_i$, $s_j$, and $s_k$.   If $s_i\cup s_k$ and $s_j\cup s_k$ are smooth, then the {\it switch equation} is $w_i+w_j=w_k$.   An invariant weight vector $\bf w$ for $\tau$ determines an oriented measured lamination $\tau(\bf w)$ carried by $\tau$.  

The {\it weight cone} $\C(\tau)$ for a train track is the convex cone of invariant weight vectors on $\tau$.  More precisely, it is the set $\bow\in \reals^p$ such that $w_i\ge 0$ for all $i$ and all the switch equations are satisfied.  The {\it  weight cell} $\PC(\tau)$ is the intersection of the weight cone with the the hyperplane $\sum w_i=1$, where the sum is over weights on 
segments of $\tau$.   We similarly define the weight cone and weight cell for branched surfaces.
\end{defns}

%\begin{defn}  A {\it half-disk} is a triple $(E,\rho, \rho')$, where $E$ is a 2-disk, $\bdry E=\rho\cup \rho'$, and  $\rho\cap  \rho'=S^0$. \end{defn}

\section {Peripheral links.}\label{Peripheral}

Rather than immediately studying arbitrary lamination links in an orientable, compact 3-manifold, it is better to begin by restricting attention to peripheral lamination links in a 3-manifold with boundary.  

\begin{defn}  A {\it peripheral oriented measured lamination link} (or a {\it peripheral link}  in an orientable, compact 3-manifold $(M,\bdry M)$ is an oriented 1-dimensional  measured lamination $(L,\mu)$ in $\bdry M$.   The link can be represented as $V_\bow(\tau)$ where $V(\tau)$ is a fibered neighborhood in $\bdry M$ of a tangentially oriented train track $\tau$.   In this context, when $\tau$ is contained in a surface $\bdry M$, $V(\tau)\embed \bdry M$ is determined by $\tau$ up to isotopy, and we can simply write $V_\bow(\tau)$ as $\tau({\bf w})$.
\end{defn}

%Note that when dealing with oriented train tracks, as in our context, it is not necessary to rule out monogons in the definition of essential train tracks.  The above definition of $\Chim$ comes from \cite{WPT:Norm}, William Thurston's paper concerning the norm on $H_2(M)$, $M$ a 3-manifold.   For a surface containing no positive Euler characteristic components, $\Chim(F)=-\Chi(F)$.  We will be dealing mostly with such surfaces.

%\begin{lemma}  Any peripheral lamination link in $(M,\bdry M)$ is carried by an oriented essential train track in $\bdry M$, and conversely any measured lamination carried by an oriented essential train track in $\bdry M$ is a peripheral link. \end{lemma}

%The proof is immediate from the theory of (unoriented) essential measured laminations in a surface.   For our purposes, we can use the lemma as a definition of ``essential measured peripheral lamination (link)."  

\begin{lemma} \label{BoundingConvexLemma}  Suppose $M$ is an compact, orientable 3-manifold with boundary.  Suppose $\tau$ is an embedded oriented train track in $\bdry M$.  The set of invariant weight vectors $\bow$ such that $\tau(\bow)$ is  carried by $\tau$ and 
bounds an oriented 2-dimensional measured lamination $(K,\nu)$ in $M$ is a convex subcone of $\C(\tau)$.
\end{lemma}

\begin{proof}  We must show that the set of invariant weights $\bow$ on $\tau$ such that $\tau(\bow)$ bounds a lamination is closed under addition and multiplication by positive scalars.   

First we show that bounding laminations carried by $\tau$ are closed under multiplication by positive scalars.   This is easy, since if $\tau(\bow)$ bounds an oriented 2-lamination, then there is an oriented branched surface $B$ with $\bdry B=\tau$, and there exists an invariant weight vector $\bov$ for $B$ whose restriction to $\tau$ is $\bow$.   This is because the measured lamination $\bdry B(\bov)=\tau(\bow)$.  Then if $\lambda>0$, $\bdry B(\lambda\bov)=\tau(\lambda\bow)$, and $\tau(\lambda\bow)$ bounds.

Next, we show that $\tau(\bow_1)$ bounds $B_1(\bov_1)$ and $\tau(\bow_2)$ bounds $B_2(\bov_2)$, then we can construct an oriented branched surface $B$ and an invariant weight vector $\bov$ for $B$ such that $\tau(\bow_1+\bow_2)$ bounds $B(\bov)$.   We may assume that $B_1$ and $B_2$ are generic, meaning that the branch locus consists of branch curves intersecting transversely at points in the interior of $M$, and $\bdry B_1=\bdry B_2=\tau$.  Now we isotope $B_1$ and $B_2$ (rel $\tau$) to a general position in the interior of $M$.   This means that $ B_1\cap B_2=\tau\cup \rho$, where $(\rho,\bdry \rho)$ is a train track properly embedded in $(M,\bdry M)$.   Finally we pinch and identify (respecting orientations) a regular neighborhood of $B_1\cap B_2$ in $B_1$ with a regular neighborhood of $B_1\cap B_2$ in $B_2$ to obtain a branched surface $B$ containing $B_1$ and $B_2$ as sub- branched surfaces.   The invariant weight vector $\bov_1$ can be regarded as an invariant weight vector for $B$, with some weights 0, and similarly for $\bov_2$.   Then $\tau(\bow_1+\bow_2)$ bounds the 2-dimensional lamination $B(\bov_1+\bov_2)$. \end{proof}

\begin{defns} With $\tau$ as in the statement of the lemma, the cone of invariant weights on $\tau$ which bound 2-dimensional measured laminations is denoted $\CB(\tau)$.   The projectivization of the cone $\CB(\tau)$ is denoted $\PCB(\tau)$.  

Suppose $M$ is an orientable 3-manifold with boundary.  An oriented surface $(F,\bdry F)\embed (M,\bdry M)$ is {\it maximal-$\Chi$} if $\Chi(F)=\max\{\Chi(G):\bdry G=\bdry F\text{ and } G \text{ has no closed components}\}$.  Viewing the oriented curve system $\bdry F$ as a peripheral link, we also say that $F$ is a {\it maximal-$\Chi$ Seifert surface} for the curve system or link $\bdry F$.

\end{defns}

When we say ``maximal-$\Chi$ Seifert surface," we usually mean the isotopy class of the maximal-$\Chi$ surface.
Saying a Seifert surface $F$ is maximal-$\Chi$ means almost the same thing as saying that $F$ is a minimal genus Seifert surface for the oriented peripheral link $\bdry F$.   However, if an oriented curve system consisting of two oriented curves bounds an oriented annulus, and the annulus is compressible, then the annulus is not maximal-$\Chi$.   Performing surgery on the compressing disk yields a maximal-$\Chi$ Seifert surface consisting of two disks.   Both the annulus and the pair of disks have minimal genus.

It is easy to show that a maximal-$\Chi$ surface $F$ is incompressible:  Suppose $D$ is a compressing disk for $F$.   Surgery on $D$ yields a surface $G$ with the same boundary and larger $\Chi$, possibly with closed components, but without sphere components.   Discarding closed components of $G$ yields another surface with the same boundary and $\Chi$ no smaller, which shows $F$ is not maximal-$\Chi$.

We shall see later, in Example \ref{NonBound}, that it is possible that an oriented curve system $C$ in $\bdry M$ does not bound an oriented surface, but when the curve system is viewed as a lamination, it bounds a oriented measured Seifert lamination.   The following lemma shows that in this case, there exists $k>1$ such that replacing $C$ by $k$ disjoint isotopic copies of itself, the resulting curve system does bound an oriented surface.

\begin{lemma}  If $\bow\in\CB(\tau)$ is an invariant weight vector with integer or rational entries, then there exists $k\in \naturals$ such that $\tau(k\bow)$ bounds an oriented surface, and bounds a maximal-$\Chi$ oriented surface.
\end{lemma}

\begin{proof}  Since $\tau(\bow)$ bounds, we know there exists $B(\bov)$ with $\bdry B(\bov)=\tau(\bow)$.  We consider the set $A$ of all $\bov\in \C(B)$ with $\bdry \bov=\bow$.  The cone $\C(B)$ is defined by linear equations and inequalities in entries $v_j$ with integer coefficients.   To obtain $A$ we  also require one equation of the form $v_j=w_i$ for each entry $w_i$ of $\bow$, setting $w_i$ equal to one entry of $\bov$.  Thus we impose more linear equations in entries of $\bov$ with rational or integer coefficients.  Therefore the set $A$ is defined by linear equations and inequalities involving only rational coefficients.  Since it is non-empty, it must contain a point $\bou$ with only rational entries.   Then there exists $k\in \naturals$ so $k\bou$ is an invariant weight vector on $B$ with integer entries, and with $\bdry B(k\bou)=\tau(k\bow)$.  We have shown that $\tau(k\bow)$ bounds some surface.   Then it also bounds a maximal-$\Chi$ surface.
\end{proof}

\begin{defn} \label{SelectDef} Suppose $(B,\bdry B)\embed (M,\bdry M)$ is a branched surface properly embedded in $M$.  Suppose $B$ carries a torus $T\embed B$ bounding a solid torus.   Suppose $T$ is embedded in $N(B)$ transverse to fibers and bounds a solid torus $\bar T$, and there exists a finite collection of meridian discs $D_i$, $i=1,2,\ldots k$ of $\bar T$, also embedded in $N(B)$ transverse to fibers, such that the projection $\pi(\cup_iD_i\cup T)=\pi(\bar T\cap V(B))$.   Then we say  $B$ {\it contains a Reeb branched surface}   $\pi(\cup_iD_i\cup T)$.

Let $\tau$ be an oriented train track embedded in $\bdry M$.  We say the oriented branched surface $(B,\bdry B)\embed (M,\bdry M)$ is {\it $\tau$-select} if the following conditions hold:

\begin{tightenumi}
\item  $\bdry B=\tau$ or $\bdry B$ is a sub- train track of $\tau$,
\item $B$ is aspherical, i.e. carries no sphere,
\item $B$ fully carries an oriented maximal-$\Chi$ surface whose orientation agrees with that of $B$,
\item $B$ has no disk of contact,
\item $\bdry_hN(B)$ is incompressible in the closure of the complement of $N(B_i)$,
\item $B$ contains no Reeb branched surface,

\end{tightenumi}

\end{defn}
Note that in the definition we do not require that $M$ or $M\setminus \intr(N(B))$ be irreducible.

The statement of the following lemma involves the notion of a ``normal surface" with respect to a fixed triangulation of a 3-manifold $M$.  This will be explained in the proof.

\begin{lemma} \label{NormalLemma} (a) Suppose $M$ is a compact, orientable, irreducible 3-manifold with boundary and $\tau\embed \bdry M$ is an oriented train track in $\bdry M$.  Then there exists a finite collection $\{B_i:  i=1,2,\ldots m\}$ of $\tau$-select branched surfaces such that if $F$ is a maximal-$\Chi$ Seifert surface for a curve system $C=\bdry F$ carried by $\tau$ then, up to isotopy,  $F$ is fully carried by one of the branched surfaces $B_i$ of the collection. 

\hip \noindent (b)  If $F$ is a maximal-$\Chi$ Seifert surface for a curve system $C=\bdry F$ carried by $\tau$ then, up to isotopy,  $F$ is fully carried by some $B_i$ as a least area  normal surface with respect to a fixed triangulation of $M$ (least area among normal surfaces isotopic to a given maximal-$\Chi$ surface $F$ with $\bdry F=C$) . 
\hip

\noindent  (c) If $L$ is any measured lamination carried by one of the branched surfaces $B_i$, then every leaf of $L$ is $\pi_1$-injective.
\end{lemma}

\begin{proof}

Choose a triangulation of $M$ such that $\tau$ is a normal train track with respect to the induced triangulation of $\bdry M$.  This means that the intersection with each 2-simplex in $\bdry M$ is one of the train tracks shown in Figure \ref{NormalTrainTrack}, up to isotopy and up to symmetries of the simplex. In particular, we require that the intersection of $\tau$ with each 2-simplex of $\bdry M$ be connected.  When choosing the triangulation, we also require that each simplex of any dimension intersect $\bdry M$ not at all or in a single simplex of dimension $\le 2$.  All of this is possible by choosing a sufficiently fine triangulation of $M$.  Suppose we are given an oriented curve system $C$ carried by $\tau$ and suppose $C$ bounds some oriented surface.  By our choice of $C$, if $C=\tau(\bow)$, then $\bow$ lies in the subcone $\CB(\tau)$ of $\C(\tau)$ representing measured laminations bounding oriented measured laminations.    Suppose $F$ is a maximal-$\Chi$ oriented surface bounded by $C$, which means it has maximal $\Chi$ among surfaces $F$ without closed components and satisfying $\bdry F=C$.  We explained earlier that $F$ is incompressible. 

\begin{figure}[H]
\centering
\scalebox{0.5}{\includegraphics{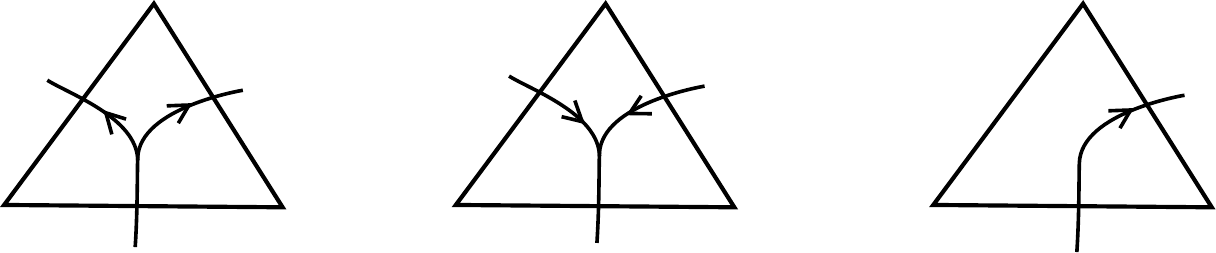}}
\caption{\footnotesize  Normal train track.}
\label{NormalTrainTrack}
\end{figure}

The fact that $F$ is incompressible allows us to use a modernized and adapted version of W. Haken's normal surface theory, \cite{WH:Normal}.
We will put $F$ into normal position with respect to the triangulation of $M$ without changing $\bdry F=C$.   As usual in normal surface theory, we begin by isotoping $F$ (rel $\bdry M$) to be in general position with respect to the triangulation, transverse to 1-simplices and 2-simplices and disjoint from vertices.   We use the usual ``combinatorial area" complexity $\gamma(F)$, equal to the number of intersections of $F$ with the 1-skeleton of the triangulation.  By our choice of triangulation and the fact that $\tau$ is oriented and that the orientation of $F$ must be compatible with that of $\tau$, we easily see that if $\sigma$ is a 2-simplex with exactly one edge in $\bdry M$, there can be no arc of $F\cap \sigma$ with both ends in $\sigma \cap \bdry M$.  As in the usual normal surface theory, a closed curve $\alpha$ of $F\cap \sigma$ innermost in $\sigma$ for any 2-simplex $\sigma$ (not contained in $\bdry M$)  can be eliminated by replacing the disk bounded in $F$ by $\alpha$ by the disk bounded in $\sigma$ by $\alpha$, then isotoping the resulting $F$ slightly to eliminate a closed curve of intersection.  If $F\cap \sigma$ contains an arc with both ends in the same 1-simplex $\rho$ not contained in $\bdry M$, then isotoping the arc to an arc in $\rho$ and a little beyond, extending the isotopy to $F$, reduces the complexity of $F$. Thus we may assume that $F$ is in normal form with respect to the triangulation.   This means that $F$ intersects 3-simplices in disks having the combinatorial type of the normal disk types shown in Figure \ref{NormalTypes}, up to symmetries of the simplex.   We now also assume that it has minimal complexity among all normal maximal-$\Chi$ surfaces $F$ with $\bdry F=C$.  

\begin{figure}[ht]
\centering
\scalebox{0.5}{\includegraphics{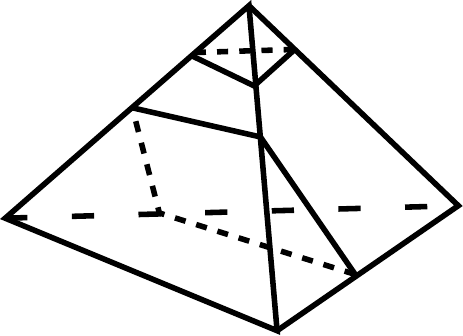}}
\caption{\footnotesize  Normal disk types.}
\label{NormalTypes}
\end{figure}

From a maximal-$\Chi$ surface of minimal area, we construct a branched surface $\hat B$ by identifying disks of $F\cap \rho$ for every 3-simplex $\rho$ of the triangulation of $M$ if they they belong to the same normal disk type, and extending the identification slightly to regular neighborhoods.  Thus $\hat B$ is the union of normal disk types that occur in $F$, appropriately joined at 2-simplices to form a branched surface.  Evidently, there are just finitely many possibilities (up to isotopy) for $\hat B$, provided we extend the locus of identification of disks of the same type in a ``standard way."   If we replace $F$ by two parallel copies of itself, then we may assume that $F$ is transverse to fibers of $N(\hat B)$ with $\bdry_hN(\hat B)\subset F$.   This means that there is a well-defined interstitial bundle for $F$ in $N(\hat B)$.

We were somewhat vague above when we said that identifications of disks of the same type should be extended in a ``standard way."   This can be made completely precise by replacing the triangulation of $M$ by a dual handle-decomposition.   Interior vertices become 3-handles, interior 1-simplices become 2-handles,  interior 2-simplices become 1-handles, interior 3-simplices become 0-handles.  A vertex on the boundary becomes a 3-handle which intersects $\bdry M$ in a 2-handle of $\bdry M$. A 1-simplex in the boundary becomes a 2-handle which intersects $\bdry M$ in a 1-handle of $\bdry M$.  A 2-simplex in the boundary becomes a 1-handle which intersects $\bdry M$ in a 0-handle of $\bdry M$. The original normal surface theory, see \cite{WH:Normal}, was explained in terms of handle-decompositions, with surfaces intersecting 0-handles, 1-handles, and 2-handles in disks belonging to various disk types.   Identifying disks of the same type in this setting yields $\hat B$.

Our goal now is to modify $\hat B$ to ensure that it has no disks of contact.  In fact, we want to show that there are finitely many ways to modify $\hat B$ such that all the normal minimal complexity surfaces carried by by $\hat B$ are also carried by one of the finitely many modified branched surfaces.   We will modify $\hat B$ by removing some interstitial bundle from $N(\hat B)$ to obtain a new branched surface neighborhood $N(B)$ of a new branched surface which has neither disks of contact nor carries spheres.   There is a description of this process in \cite{WFUO:IncompressibleViaBranched}; we will give essentially the same argument here, somewhat reorganized.

\hop
\noindent {\it Claim 1: If $\hat B$ carries a sphere and fully carries the maximal-$\Chi$ surface $F$, then $\hat B$ has a disk of contact $E$ in $F$. }

Suppose $\hat B$ carries a sphere $S$, transverse to fibers of $N(\hat B)$ and transverse to the surface $F$ carried by $\hat B$.  Assume, for now, that both $S$ and $F$ are disjoint from $\bdry_hN(\hat B)$.   We will show that we can replace $S$ with another sphere $S'$ transverse to fibers with fewer curves of intersection with $F$.   Since $F$ is incompressible and two-sided, it is also $\pi_1$-injective, which means that any curve of $F\cap S$ bounds a disk in $F$.   Choose a curve $\alpha$  of $F\cap S$ innermost on $F$.   Then $\alpha$ bounds a disk $D$ in $F$  and it bounds two disks $H$ and $H'$, with disjoint interiors  in $S$.  Then either $H\cup D$ or $H'\cup D$ is an embedded sphere $S'$ carried by $\hat B$.   After a small isotopy to push $S'$ off of $D$, $S'$ has fewer curves of intersection with $F$ than $S$.  Repeating, we obtain a sphere transverse to fibers and disjoint from $F$, which we will again call $S$.   

By irreducibility of $M$, $S$ bounds a ball $K$ in $M$, and $F$ does not intersect $K$.  
We may now assume that both $S$ and $F$ are embedded transverse to fibers in $N(\hat B)$, with $\bdry_hN(\hat B)\subset (S\cup F)$.  
If there is no interstitial bundle for $S\cup F$ in $K$, then $S\subset \bdry_hN(\hat B)$.  Since $F$ is also fully carried and can be isotoped so $\bdry_hN(\hat B)\subset F$, we conclude $F$ has a sphere component, contrary to assumption.

If there is interstitial bundle for $S\cup F$ in $K$, let $J$ be the (total space of) the interstitial bundle.   Recall $\bdry_vJ$ denotes the ``vertical boundary" of $J$, which is a collection of annuli.   Then $\bdry(\bdry_vJ)$ is a curve system in $S$.   The curve system divides $S$ into regions which are colored alternately white if contained in $\bdry_hN(\hat B)$, and black if intersects $J$.   We claim that there is a white region $W$ which is a planar surface with more than one boundary, i.e. not a disk.   If all white regions were disks, then for every annulus $A$ of $\bdry_vJ$ there would be two disjoint white disks $D_0$ and $D_1$ adjacent to $A$ and $D_0\cup D_1\cup A$ is a sphere bounding a ball in $K$.   If we fill each of these balls with a product bundle extending $J$ to a bundle $\bar J$, say, then $\bar J$ gives an $I$-bundle structure to the ball $K$, with $\bdry_hJ=S=\bdry K$, which is impossible.  We conclude that there is at least one white region $W\subset \bdry_hN(\hat B)$ which is a planar surface with at least two boundary components.

Now we discard $S$ and isotope $F$ such that $\bdry_h\hat B\subset F$, and in particular $W\subset F$.  Let $\alpha$ be a curve of $\bdry W$, then because $F$ is incompressible, $\alpha$ bounds a disk $D$ in $F$.   Either $D$ is a disk of contact in $F$ or $D\supset W$.   In the latter case, $D$ must contain a disk of contact in $F$ with boundary any of the other curves of $\bdry W$. 

%Now let $J$ denote the interstitial bundle for $F$ in $N(\hat B)$ (rather than for $F\cup S$).  We see that the component $J_0$ of $J$ containing $A_j$ must be a product of the form $P\times I$, where $P$ is planar, with $P\times \{0\}\subset D_{j0}$ and $P\times \{1\}\subset D_{j1}$.  There may be other components of $J$ in $K$.   

This completes the proof of Claim 1.

\hop
\noindent {\it Claim 2: If $\hat B$ has a disk of contact and fully carries the minimal complexity, maximal-$\Chi$ surface $F$, $F$ transverse to fibers of $N(\hat B)$ with $\bdry_hN(\hat B)\subset F$, then it has a disk of contact $E\subset F$.   Further, if $\bdry E\subset A$, where $A$ is an annular component of $\bdry_vN(\hat B)$, then the component of the interstitial bundle for $F$ in $N(\hat B)$ containing $A$ has the form $P\times I$, where $P$ is a planar surface and $P\times \{0\}\subset E$.} 

Suppose $E$ is a disk of contact for $\hat B$.   This means $E$ is a disk embedded in $N(\hat B)$ transverse to fibers with $\bdry E\subset \intr(\bdry_vN(\hat B))$.   We suppose that $F$ is also embedded in $N(\hat B)$ transverse to fibers.   We assume $\bdry_hN(\hat B)\subset F$ and $E$ is disjoint from the horizontal boundary.   We will show that we can replace $E$ with another disk of contact $E'$ transverse to fibers with fewer curves of intersection with $F$.   Since $F$ is $\pi_1$-injective, any curve of $F\cap S$ bounds a disk in $F$.  Choose a curve $\alpha$  of $F\cap E$ innermost on $F$.   Then $\alpha$ bounds a disk $D$ in $F$  and it bounds a disk $H$ in $E$.  Either $H\cup D$ is an embedded sphere $S$ carried by $\hat B$ or $(E\setminus H)\cup D$ is another disk of contact $E'$ with the same boundary as $E$.   In the latter case, after a small isotopy, $E'$ has fewer curves of intersection with $F$ than $E$.  In the former case, by Claim 1, $\hat B$ has a disk of contact in $F$.

Repeating the argument, we obtain a sphere carried by $\hat B$, and by Claim 1 we immediately get a disk of contact in $F$, or we get a  disk of contact with fewer curves of intersection with $F$.   Repeating as often as necessary, we either get a disk of contact in $F$ using Claim 1, or we  end with a disk of contact (transverse to fibers in $N(\hat B)$) disjoint from $F$.   If $E$ is such a disk of contact disjoint from $F$, then $\bdry E$ is in the interior of an annulus component $A$ of $\bdry_vN(\hat B)$.   From $E\cup A$, we obtain a null-homotopy for each curve of $\bdry A$ (which is an embedded curve in $F$) in the manifold $M|F$ obtained by cutting $M$ on $F$.  By the incompressibility of $F$, each component of $\bdry A$ must bound a disk in $F$, say the boundary components of $\bdry A$ bound disks $D_0$ and $D_1$ in $F$.  If $D_0\cup E$ or $D_1\cup E$  (together with an annulus in $A$) yields a sphere carried by $\hat B$, then $F$ must contain a closed component contained in the ball $K$ bounded by the sphere, see the schematic in Figure \ref{ClosedFig}.  This contradicts the minimal complexity of $F$. (The ball $K$ could be on the other side of $D_0\cup E\cup A$ in the figure.)  Thus both $D_0$ and $D_1$ are disks of contact in $F$, and there must be a  component of the interstitial bundle in the ball bounded by $D_0\cup D_1\cup A$ which is a product of the form $P\times I$, where $P$ is planar and $A\subset \bdry P\times I$.  

\begin{figure}[ht]
\centering
\scalebox{0.5}{\includegraphics{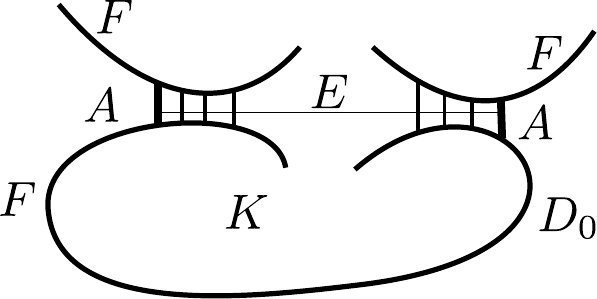}}
\caption{\footnotesize  Component of $F$ in a ball.}
\label{ClosedFig}
\end{figure}

In the above argument, if we obtain a disk of contact $D_0$ in $F$ from Claim 1, and $\bdry D_0\subset A$ where $A$ is an annulus of $\bdry_vN(\hat B)$, then the other boundary component of $A$ must also bound a disk $D_1$ in $F$.  Using an argument very similar to the one in the previous paragraph, we conclude $D_1$ is also a disk of contact in $F$, and there must be a  component of the interstitial bundle in the ball bounded by $D_0\cup D_1\cup A$ which is a product of the form $P\times I$, where $P$ is planar and $A\subset \bdry P\times I$. 

This completes the proof of Claim 2.

 \hop
 
 \noindent {\it Claim 3:  If $E$ is a disk of contact in $F$, then $E$ has minimal area among disks $G$ with $\bdry G=\bdry E$ in general position with respect to the triangulation.}
 
 Suppose $E$ does not have minimal area.  Let $A$ be the annular component of $\bdry_vN(B)$ containing $\bdry E$.   We suppose $E\subset F$.   Let $G$ be a disk with $\bdry G\subset \intr(A)$, so $\bdry G\cap F=\emptyset$, and $\gamma(G)<\gamma(E)$.  We isotope $G$ so it is transverse to $F$.  If $G$ is disjoint from $F$, then if $\bdry A\subset F$, we can isotope $\bdry G$ to $\bdry E \subset F$, and we see that $F$ does not have minimal complexity, since we could replace $E\subset F$ by $G$.  In general, if $G$ intersects $F$, let $\alpha$ be a closed curve of intersection innermost in $F$, bounding a disk $D$ in $F$ and bounding a disk $H$ in $G$.  Then $\gamma(D)\le \gamma(H)$ since $F$ has minimal area.  Isotoping $H$ to $D$ and slightly beyond, we obtain a new disk $G'$ with $\gamma(G')\le \gamma(G)$.   Also $G'$ intersects $F$ in fewer curves than $G$.   Repeating this argument, we finally obtain a disk $G$ disjoint from $F$ with the same boundary as the original $G$ and with $\gamma(G)<\gamma(E)$.   This contradicts the minimal complexity of $F$ and completes the proof of Claim 3.
 
 \hop

As we explained above, now for each $\hat B$ constructed from a normal minimal area maximal-$\Chi$ surface $F$, there is a finite collection of minimal area disks of contact in $F$, such that splitting on all of these disks of contact gives a branched surface $B$ carrying $F$.  More precisely, we saw that for every disk of contact $E$ in $\hat B$ with boundary in an annulus $A\subset \bdry_v\hat B$, there exist disks $D_0$ and $D_1$ in $F$ such that $A\cup D_0\cup D_1$ is a sphere bounding a ball containing a component of interstitial bundle for $F$ in $N(\hat B)$ of the form $P\times I$, where $P$ is planar and $A\subset \bdry P\times I$.   We split to eliminate the disc of contact by removing $P\times I$ and other components of interstitial bundle in $K$ from the interstitial bundle.  For a fixed $\hat B$, considering all possible least area maximal-$\Chi$ surfaces $F$ fully carried by $\hat B$, there are finitely many ways to eliminate all disks of contact by splitting as above,  because for each disk of contact and $F$, the corresponding disks of contact $D_0$ and $D_1$ in $F$ are least area disks, as we proved in Claim 3, so there are finitely many possibilities for $D_0$ and $D_1$.  We eliminate all disks of contact in this way.   So there are finitely many normal branched surfaces $B$ without disks of contact which fully carry all the minimal area surfaces $F$ carried by $\hat B$.   We showed that every maximal-$\Chi$ surface $F$ with $\bdry F$ carried by $\tau$ is carried by one of finitely many $\hat B$'s, so the end result is that we have finitely many  branched surfaces $B_i$ such that every maximal-$\Chi$ surface $F$ is fully carried by one of the $B_i$.  Further, each $B_i$ has no disk of contact, and does not carry a sphere.  

We first verify that each $B_i$ satisfies condition (v) in Definition \ref{SelectDef}, of a $\tau$-select branched surface.   To prove that $\bdry_hN(B_i)$ is incompressible in the (completion of the) complement of $N(B_i)$, we suppose $F$ is a minimal area, maximal-$\Chi$, incompressible surface fully carried by $B_i$, embedded transverse to fibers in $N(B_i)$ with $\bdry_hN(B_i)\subset F$.   If $D$ is a compressing disk for $\bdry_hN(B_i)$ in the completion of the complement of $N(B_i)$, then $D$ is also a potential compressing disk for $F$, which implies there exists a disk $D'\subset F$ with $\bdry D'=\bdry D$.   If $D'$ is not contained in $\bdry_hN(B_i)$, then clearly $D'$ contains a disk of contact, which is a contradiction.

The goal now is to show that each $B_i$ is orientable such that the orientation agrees with the orientations of the maximal-$\Chi$ surfaces. We use ideas from \cite{UO:Homology}.  In the cited paper, oriented surfaces are assumed to have minimal $\Chim$ among oriented surfaces representing the same homology class, which is a little different from our context, but the ideas carry through.

\begin{figure}[ht]
\centering
\scalebox{0.5}{\includegraphics{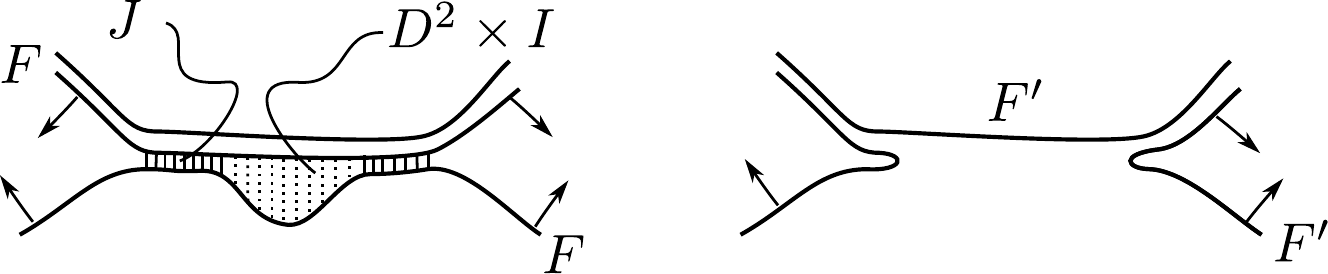}}
\caption{\footnotesize  Replacing $F$ by $F'$.}
\label{Swap}
\end{figure}

Let $B$ denote one of the $B_i$, and we suppose $F$ is a minimal area maximal-$\Chi$ surface fully carried by $B$, with $\bdry_hN(B)\subset F$.   We let $J$ denote the interstitial bundle for $F$.  We first observe that no component of $J$ can have the form $P\times I$ where $P$ is planar and $P\times \{0\}$ is contained in a disk $D_0$ in $F$.  In this case $B$ would have a disk of contact.

Before proceeding, we define an augmented interstitial bundle $\bar J$.   The bundle $\bar J$ is obtained from $J$ by including the ball components of the completion of $M\setminus N(B)$, which can be given a product structure of the form $D^2\times I$ with $D^2\times \bdry I$ in the horizontal boundary and $\bdry D^2\times I$ being a product component of the vertical boundary.   We have filled in $\hbox{disk}\times\hbox{interval}$ ``bubbles." 

  Now we claim that the orientations of $F$ at opposite ends of fibers of the interstitial bundle for $F$ in $N(B)$ must be consistent, to yield an orientation for $B$.  If not, we consider a component $\bar J_0$ of the augmented interstitial bundle with the property that the orientations of $F$ at opposite ends of the fiber are inconsistent.     A key observation in our situation is that $\bar J_0$ cannot intersect $\bdry M$, otherwise we have inconsistent orientations of $C=\bdry F$ carried by $\tau$.  (Recall that we chose $C$ to be carried as an oriented curve system obtaining its orientation from the orientation of $\tau$.)  Clearly $\bar J_0$ cannot be a product of the form $\hbox{disk}\times I$, since this would mean $B$ has a disk of contact.   We modify $F$ as shown in Figure \ref{Swap} by replacing $\bdry_h\bar J_0$ by $\bdry_v \bar J_0$.   This yields a new oriented surface $F'$ which has the same boundary exactly, and which can be put in normal position, although it is not normal.  If we show that $F'$ contains no sphere components and $\Chi(F')>\Chi(F)$ or $\Chi(F')=\Chi(F)$ and the complexity (area) of $F'$ can be made smaller than that of $F$ when $F'$ is isotoped to normal form, then we have a contradiction.  First we show that $F'$ contains no sphere components.  Suppose $F'$ contains a sphere component $S$, then $S$ contains at least one annulus of $\bdry_v\bar J_0$.   A curve of $\bdry (\bdry_v\bar J_0)$ innermost on $S$ yields a disk $D_0$ in $\bdry_hN(B)$.  The curve $\bdry D_0$ is one boundary curve of an annulus of $\bdry_vN(B)$.    Using property (v), that $\bdry_hN(B)$ is incompressible in the complement of $N(B)$, we know that the boundary of the potential compressing disk $D_0\cup A$ of the horizontal boundary must bound a disk $D_1$ in $\bdry_hN(B)$.  Then the sphere $D_0\cup A\cup D_1$ bounds a ball which can be given the structure of a product $D^2 \times I$.  This is a contradiction since we filled all complementary bubbles of this form to obtain $\bar J$.  The same argument actually shows that $F'$ does not contain any ``new" disk components intersecting $\bdry_vJ_0$.  Now it remains to show that either $\Chi(F')>\Chi(F)$  or $F'$ can be isotoped to reduce complexity.   If $\bar J_0$ is a bundle over an annulus, $\Chi(F')=\Chi(F)$, but clearly $F'$ can be isotoped to a normal surface with smaller area than $F$.  Otherwise, $\Chi(\bdry_h\bar J_0)<0$, so $\Chi(F')>\Chi(F)$, which contradicts our choice of $F$. 

We conclude that  $F$ is fully carried by one of finitely many {\it oriented} normal branched surfaces $B_i$ obtained from possible $\hat B$'s by splitting on finitely many least area disks of contact of each $\hat B$.  By construction, no $B_i$ has a disk of contact.   By Claim 1, no $B_i$ carries a sphere.

It remains to verify that every $B_i$ satisfies condition (vi).   To show that $B_i$ contains no Reeb branched surface, suppose $B_i$ carries a torus $T\embed B$ bounding a solid torus $\bar T$ with $T$ is embedded in $N(B)$ transverse to fibers which bounds the solid torus $\bar T$.  Suppose there exists a finite collection of meridian discs $D_i$, $i=1,2,\ldots k$ of $\bar T$, also embedded in $N(B)$ transverse to fibers, such that the projection $\pi(\cup_iD_i\cup T)=\pi(\bar T\cap N(B))$.   Then  $\pi(\cup_iD_i\cup T)$ is a Reeb branched surface $R$.  Note that $R$ is a branched surface in $\bar T$ which contains $\bdry \bar T=T$.   We can assume that $\bdry_hN(B_i)\cap \bar T\subset \cup_i D_i$ by replacing each $D_i$ by two isotopic copies of itself if necessary.   Then any component of $\bdry_hN(B_i)\cap \bar T$ is contained in some $D_i$, hence must be a disk.  It then follows from the incompressibility of $\bdry_hN(B_i)$ that every complementary component of $N(B_i)$ in $\bar T$ must have the form $E\times I$, where $E$ is a disk.   Let us suppose, without loss of generality, that the induced transverse orientation of $T$ is outward with respect to $\bar T$, otherwise we can reverse the orientation.   Let us also suppose that the disks $D_i$ are numbered such that isotoping $D_i$ in the direction of the transverse orientation across complementary products we can move $D_i$ such that it contains $D_{i+1}$.   Then isotoping all the $D_i$ in this way, we increase the combinatorial area of $\cup_i D_i$ by the area of $T$.   Isotoping in the opposite direction,  we decrease the area by the area of $T$.  Now suppose $F$ is a minimal area maximal-$\Chi$ surface fully carried by $B_i$.   Performing the same isotopies of disks in $F$ across complementary product regions in the sense opposite to the orientation, we can decrease the area of $F$ by the area of $T$.  This contradicts the fact that $F$ is a minimal area maximal-$\Chi$ surface.    

The proof of the $\pi_1$-injectivity of any leaf of any measured lamination $L$ carried by a $B_i$ follows below.
\end{proof}

We want to claim that the branched surfaces $B_i$ described in Lemma \ref{NormalLemma} have the property that any measured lamination carried by $B_i\embed M$ has $\pi_1$-injective leaves.   For this purpose, we state a modified version of Theorem 2.11 from  \cite{UO:MeasuredLaminations}.   The modification of the theorem has weaker hypotheses and also weaker conclusions, but the proof can be obtained directly from the previous paper by ignoring appropriate hypotheses and conclusions.  The difference is that here we are not interested in $\bdry$-incompressiblity of leaves, or injectivity of relative fundamental groups of leaves.  For stating the theorem, we introduce another condition on a branched surface:  We say $D$ is a {\it monogon} for the branched surface $B\embed M$ if $D\embed M$, $D\cap N(B)=\bdry D$, $D\cap \bdry_hN(B)$ is a closed arc of $\bdry D$ embedded in $\bdry_hN(B)$, and the complementary arc is the intersection of a fiber of $N(B)$ with $\bdry_vN(B)$.  It is obvious that an orientable branched surface $B\embed M$ has no monogons, but the following theorem also applies to non-orientable branched surfaces.

\begin{thm}  Suppose $B\embed M$ has the following properties:

\begin{tightenumi}
\item No component of $\bdry_hN(B)$ is a sphere,
\item $B$ has no disk of contact,
\item $\bdry_hN(B)$ is incompressible in the completion of $M\setminus N(B)$,
\item  there are no monogons for $B$,
\item $B$ contains no Reeb branched surface,
\end{tightenumi}

Suppose $L$ is any measured lamination carried by $B$.  If $\ell$ is a leaf of $L$, then the map $\ell\to M$ induces an injection $\pi_1(\ell)\to \pi_1(M)$.  
\end{thm}

As we mentioned, the proof is an easy adaptation of the proof of the proof of Theorem 2.11 in  \cite{UO:MeasuredLaminations}.   The difficult part of the theorem is showing that the leaves of a measured lamination carried by $B$ are $\pi_1$-injective {\it even if the lamination is not fully carried by $B$}.

The following corollary follows since the $\tau$-select branched surfaces $B_i$ guaranteed by Lemma \ref{NormalLemma} satisfy condition (i) above because they are aspherical, and they satisfy (iv) because they are orientable.

\begin{corollary} If $B$ is  a $\tau$-select branched surface, then any measured lamination carried by $B$ has $\pi_1$-injective leaves.
\end{corollary}

 %Although we know that there exist finitely many $\tau$-select branched surfaces which carry all maximal-$\Chi$ Seifert surfaces for curve systems carried by $\tau$, we will give examples to show that not every surface carried by a $\tau$-select branched surface is a maximal-$\Chi$ Seifert surface.   There is a torsion phenomenon which makes it impossible to prove a lemma like the ``elementary lemma," Lemma 1 in W. Thurston's paper , \cite{WPT:Norm}, about the norm on homology.   In our setting, it is preferable to work with certain desirable Seifert {\it laminations} rather than maximal-$\Chi$ Seifert surfaces.  
 
 The following is a technical lemma we will use more than once.
 
 \begin{lemma} \label{TechLemma} Suppose we are given an oriented $\tau\embed \bdry M$ as in Lemma \ref{NormalLemma}, and a triangulation as in the proof.  Suppose $\{F_n\}$ is a sequence of oriented surfaces without sphere components with $\bdry F_n$ carried by $\tau$, $\bdry F_n=\tau(\boy_n)$.  Let $K_n$ be the sum of the entries of $\boy_n$ and let $\bow_n=\boy_n/K_n$, normalized so $\bow_n\in \PC(\tau)$.  Suppose $\bow_n\to \bow$ and suppose $\limsup \Chi(F_n)/K_n=P$.  Then there exists a $\tau$-select normal branched surface $B_i$ in the collection guaranteed by Lemma \ref{NormalLemma} and a measured lamination $B_i(\bov)$ such that $\bdry B_i(\bov)=\tau(\bow)$ and $\Chi(B_i(\bov))= P$.
 \end{lemma}
 
 \begin{proof}  We begin by replacing $F_n$ by a surface of minimal area (with respect to the triangulation) among all maximal-$\Chi$ normal surfaces with the same boundary.  In particular, we discard all closed components of $F_n$.
 For the new sequence $\{F_n\}$ we may have increased $\limsup \Chi(F_n)/K_n=P$.  By passing to a subsequence we can assume that $F_n$ satisfies $\lim_{n\to \infty}\Chi(F_n)/K_n=P$.  By passing to a subsequence again, we may suppose that that every $F_n$ is fully carried by one of the normal $\tau$-select branched surfaces, say $B_i$, so $F_n=B_i(\boz_n)$, where $\boz_n$ has integer entries.   Now we normalize $\boz_n$ by dividing by the sum $L_n$ of the entries of $\boz_n$, obtaining $\bou_n=\boz_n/L_n\in \PC(B_i)$.  Since $ \PC(B_i)$ is compact, by passing to a subsequence, we can assume $\bou_n$ converges, $\bou_n \to \bou\in \PC(B_i)$.   
 
 We let $p$ be the number of sectors of $\tau$.   Since $K_n$ is the sum of the the entries of $\boy_n=\bdry \boz_n$ (some of the entries of $\boz_n$, with at most $p$ repetitions each), we have $K_n\le pL_n$, hence $0< K_n/L_n\le p$.  (A sector of $B_i$ may intersect $\tau$ in more than one, but not more than $p$, sectors of $\tau$.)  Passing to a subsequence again, we arrange that $K_n/L_n$ converges, $K_n/L_n\to \lambda$, where $0\le\lambda\le p$.  Suppose for now that $\lambda\ne 0$.   Restricting to the boundary, $\bdry\boz_n/K_n\to \bow$, so $\bdry\boz_n/L_n=(\bdry\boz_n/K_n)(K_n/L_n)\to \lambda\bow$.  Then since $\bdry\boz_n/L_n\to \lambda\bow$, $\bdry\boz_n/(\lambda L_n)\to\bow$.   In this case we can easily conclude that  $\Chi(B_i(\bou/\lambda))\ge \limsup \Chi(F_n/K_n)$ using the following equations:

$$\Chi(B_i(\boz_n))= \Chi (F_n)$$
$$\Chi\left(B_i\left(\frac {\boz_n}{K_n}\frac{K_n}{L_n}\right)\right)= \left(\frac {K_n}{L_n}\right)\Chi(F_n)/K_n.$$
\noindent  Taking limits as $n\to \infty$ we get:
$$\Chi(B_i(\bou))= \lambda  \lim \Chi(F_n/K_n) $$
$$\Chi(B_i(\bou/\lambda))=  \lim \Chi(F_n/K_n)$$

\noindent Further, $\bdry B_i(\bou/\lambda)=\tau(\bow)$, so taking $\bov=\bou/\lambda$, we have $\bdry B(\bov)=\tau(\bow)$.  This shows that $ \lim \Chi(F_n/K_n)=\Chi(B(\bov))$ and $\Chi(B(\bov))$ is finite in this case.

It remains to rule out the possibility that $\lambda=0$.   If $\lambda=0$,  $\bdry\boz_n/L_n=(\bdry\boz_n/K_n)(K_n/L_n)\to \lambda \bow=0$, which means $B_i(\bou)$ is a closed lamination.   Eliminating sectors of $B_i$ assigned 0 entries by $\bou$ we obtain an invariant weight vector $\bou'$ on a closed sub- branched surface $B_i'$ of $B_i$.   Since $\C(B_i')$ is non-empty, $B_i'$ carries a closed surface $S=B_i(\boz)$ for some invariant weight vector $\boz$ with integer entries.  The weights induced by $F_n'$ on $B_i'$ must approach infinity as $n\to\infty$, otherwise the normalized weights $\boz_n/L_n$ on $\tau$  cannot approach $0$.    Thus for sufficiently large $n$ each entry of $\boz$ is smaller than the corresponding entry of $\boz_n$.   In other words, the  $j$-th entry $z_j$ of $\boz$ is less than the $j$-th entry $z_{nj}$ of $\boz_n$ (for the fixed large $n$).  This means $B_i(\boz_n-\boz)$ is a surface $G$ with $\bdry F_n'=\bdry G$.   Since $B_i$ does not carry spheres, $\Chi(S)\le 0$.   If $\Chi(S)<0$, then $\Chi(G)>\Chi(F_n')$ and $F_n'$ is not a maximal-$\Chi$ surface, a contradiction.   If $\Chi(S)=0$, then   $\Chi(G)=\Chi(F_n')$ but $\gamma(G)=\gamma(F_n')-\gamma(S)<\gamma(F_n')$, a contradiction to our choice of $F_n'$ as a surface of minimal area among all maximal-$\Chi$ normal surfaces with the same boundary as $F_n$. 
 \end{proof}

\begin{defn}\label{TautDef}  In this definition we do not require that $M$ be irreducible.  Suppose $\tau\embed \bdry M$ is an oriented train track and $\tau(\bow)$ represents a peripheral lamination link $(L,\mu)\embed \bdry M$ which bounds some 2-dimensional Seifert lamination.  If $B\embed M$ is a properly embedded, oriented branched surface in $M$, we say $B$ is {\it aspherical} if it carries no sphere.   We define 
$$X((L,\mu)):=\sup \{\Chi (B({\bov})): B \text{ oriented, aspherical and } \bdry B(\bov)=(L,\mu)\}.$$

We define a function $X$ on $\CB(\tau)$ as follows:  $$X(\bow)=X(\tau({\bow})):=\sup \{\Chi (B({\bov})): \  \bdry B(\bov)=\tau(\bow) \hbox{ and } B\embed M \hbox{ is aspherical } \}.$$
 
A Seifert lamination $B(\bov)$ for $\tau(\bow)$ in $\bdry M$ is {\it maximal-$\Chi$} if $\Chi(B(\bov))=X(\tau(\bow))$ and $B$ is aspherical.
Now suppose $C=\tau(\bow)$ is an oriented curve system.   If $\bdry F=C$ we say $F$ is a {\it taut surface} or {\it taut Seifert surface} if it has no closed components and $\Chi(F)=X(\tau(\bow))$.
\end{defn}

\begin{remarks}  (1) The supremum in the above definition of $X$ is finite, by Lemma \ref{TechLemma}.

(2) The difference between a maximal-$\Chi$ Seifert surface $F$ for a curve system $C\embed \bdry M$ and a taut Seifert surface is that the former is maximal-$\Chi$ among Seifert surfaces, while the latter is maximal-$\Chi$ among {\it measured Seifert laminations}  when the boundary is viewed as a lamination link.
\end{remarks}

The following example shows that an maximal-$\Chi$ Seifert surface may not be taut, i.e. not maximal-$\Chi$ when viewed as a Seifert lamination.   This means that there is an example of a maximal-$\Chi$ Seifert surface $F$ such that $X(\bdry F)>\Chi(F)$; or, in other words, we can replace $F$ by a ``better" measured lamination $B(\bov)$, where $\bdry B(\bov)=\bdry F$ as a measured lamination, $B$ is aspherical, and $\Chi(B(\bov))>\Chi(F)$.

 \begin{ex} Let $V$ be a solid torus, $H$ a sphere with 3 holes, $G$ a genus two surface with one hole.  Let $\alpha$ be an oriented longitude curve in $\bdry V$.   Let $M$ be a 3-manifold constructed from $V$, $G\times I$ and $H\times I$ as follows.  Suppose $\alpha_i$, $i=0,1,2,3$, are disjoint, embedded, oriented closed curves in $\bdry V$, each isotopic to and disjoint from $\alpha$.   Let $A_i$, $i=0,1,2,3$, be disjoint product regular neighborhoods of $\alpha_i$, $i=0,1,2,3$ which are also disjoint from $\alpha$.   We let $N=N(\alpha)$ be another annulus representing a fibered neighborhood of the train track $\alpha\embed \bdry V$, where $N$ is disjoint from the annuli $A_i$.  Now identify $\bdry G\times I$ with $A_0$ and identify components of $\bdry H\times I$ with $A_i$, $i=1,2,3$.   For attaching maps on $\bdry H\times I$, first choose an orientation of $H$ and perform the identifications such that the attaching curves have consistent orientations in $\bdry V$.  The result is an orientable 3-manifold containing disjointly embedded annuli $A_i$, $i=0,1,2,3$, such that cutting on all the annuli yields the union of $V$, $H\times I$ and $G\times I$.  We regard $\alpha$ as a train track in $\bdry M$ and $N=N(\alpha)$ as a fibered neighborhood of $\alpha$, and we also regard $\alpha$ as a curve system carried by $\alpha$.  We claim that an oriented surface $G\times \{1/2\}$ extended by an annulus in $V$ from $\alpha$ to $\bdry G$  yields a maximal-$\Chi$ Seifert surface for $\alpha$.   Abusing notation, let us denote as $G$ this extended surface.   To show that $G$ is maximal-$\Chi$, we suppose $F$ is an maximal-$\Chi$ surface bounded by $\alpha$ and show that $F$ is isotopic to $G$.   We know $F$ is incompressible, so we begin by simplifying intersections of $F$ with $A=\cup_i A_i$.   After suitable isotopy, we can assume that intersection curves of $F\cap A$ are essential closed curves.   Next, it is not difficult to show that $F\cap (G\times I)$, and $F\cap (H\times I)$ are incompressible and therefore horizontal, i.e. isotopic to a surface transverse to interval fibers of the products.  We analyze $F\cap V$, which is an incompressible surface in $V$ with boundary equal to $\alpha$ union some essential curves in $A\subset\bdry V$.   We conclude that  $F\cap V$ must consist of annuli.  In particular, there is one annulus of $F\cap V$ containing $\alpha$.   
 
We choose orientations for $G$ and $H$, then we let $g$ be the algebraic number of connected horizontal surfaces (homeomorphic to $G$) in $F\cap (G\times I)$ (so that horizontal surfaces with opposite orientations cancel.)  Similarly we let $h$ be the algebraic number of connected horizontal surfaces in $F\cap (H\times I)$.  On the one hand, regarding $F$ as a relative cycle in $H_2(M,V)$, we have $\bdry[F]=[\alpha]$.  On the other hand, $g[G]+h[H]$ represents the same relative cycle and $\bdry(g[G]+h[H])=g[\alpha]+3h[\alpha]$, so we must have $h=0$, $g=1$.   If there are any cancellations of horizontal surfaces, clearly $F$ does not have maximal $\Chi$, so we conclude $F$ does not intersect $H\times I$ and intersects $G\times I$ in a single connected horizontal surface homeomorphic to $G$.   We conclude $F$ is isotopic to $G$.
  
We now claim $G$ is not taut.   To show it is not taut, we observe that there is an oriented surface surface $S$ consisting of $H\times \{1/2\}$ extended by 3 annuli in $V$, each with one boundary transverse to fibers of $N=N(\alpha)$.   Then the weighted surface $(1/3)S$ (a measured lamination) satisfies $\Chi((1/3)S)=(1/3)\Chi(H)=-1/3>\Chi(G)=-3$.

We have shown $G$ is a maximal-$\Chi$ Seifert surface, but it is not a taut Seifert surface.
 \end{ex}
 
 \begin{ex} \label{NonBound} We modify the previous example to show that a peripheral curve system may bound a Seifert measured lamination, but not a Seifert surface.   In this example we construct $M$ just as before, but we do not attach $G\times I$.   Thus $M$ is constructed from $V$ with $H\times I$ attached as before.   Here the identifications are along three annuli $A_i$, $i=1,2,3$, mutually disjoint and disjoint from $N=N(\alpha)$ and we let $A=\cup_iA_i$.  Suppose $\alpha$ bounds an oriented surface $F$.  We may assume that $F$ is a maximal-$\Chi$ surface.  As above, $F$ intersects $H\times I$ in horizontal surfaces.   If $R=F\cap (H\times I)$,  then because the three boundary components of each component of $R$ are consistently oriented, regarding $[R]$ as a relative cycle in $H_2(M,V)$ we have $\bdry[R]=3k[\alpha]$ in $H_1(V)$ for some integer $k$.  Now, considering $[F]\in H_2(M,V)$, we have $\bdry[ F]=[\alpha]$ in $H_2(V)$.   So $[\alpha]=3k[\alpha]$.  This is a contradiction.
 
 Thus the curve system consisting of the single curve $\alpha$ (carried by the train track $\alpha$) does not bound a Seifert surface, but it does bound a Seifert measured lamination, as in the previous example.
 \end{ex}

A question which immediately comes to mind is whether the supremum in the definition of $X(\bf{w})$ is achieved by a 2-dimensional measured lamination $B(\bf{v})$.   In other words, does a peripheral lamination link which bounds a Seifert lamination bound a maximal-$\Chi$ Seifert lamination?  We shall prove this in the next theorem.  

\begin{thm}\label{AchieveX}  Let $(L,\mu)$ be an oriented measured lamination in $\bdry M$ and suppose it is fully carried by an oriented train track $\tau$ so $(L,\mu)=\tau(\bow)$.   If $(L,\mu)$ bounds a Seifert lamination, i.e. $\bow\in\CB(\tau)$, then there exists a maximal-$\Chi$ Seifert lamination $(\Lambda,\mu)$ with $\bdry (\Lambda,\mu)=\tau(\bow)$ and $\Chi((\Lambda,\mu))=X(\tau(\bow))=\sup\{\Chi(B(\bov):\bdry B(\bov)=\tau(\bow),\ B\ \text{aspherical}\}$.   $X((L,\mu))$ is finite.  The lamination $(\Lambda,\mu)$ is carried (but not necessarily fully carried) by one of the finite collection $\{B_i\}$ of the $\tau$-select branched surfaces of Lemma \ref{NormalLemma}.
\end{thm}

\begin{proof} Without loss of generality $\bow\in \PCB(\tau)$, which means the sum of its entries is 1.  There exists a sequence of 2-dimensional laminations $C_n(\bov_n)$ with $\bdry C_n(\bov_n)=(L,\mu)$ and $C_n$ an aspherical branched surface such that $\Chi(C_n(\bov_n))\to X(L,\mu)$.  By splitting $C_n$ sufficiently along leaves of $C_n(\bov_n)$, we can assume $\bdry C_n$ is a splitting of $\tau$, which means any lamination carried by $\bdry C_n$ is carried by $\tau$.  Splitting preserves the aspherical property of $C$, so $C_n$ is still aspherical.  

There is a composition boundary map $b_n$ of linear maps $b_n:\C(C_n)\xrightarrow{\bdry} \C(\bdry C_n)\to\C(\tau)$ which transforms a weight vector on $ \C(C_n)$ to an induced weight vector on $\tau$.  The linear maps have integer coefficients, so they take vectors with rational entries to vectors with rational entries.   We wish to work with $\PC(\tau)$, the weight cell for $\tau$, which is identified with the subspace $\PC(\tau)\subset \C(\tau)$ consisting of invariant weight vectors with the property that the sum of the entries is $1$.   We consider the set of invariant weight vectors in the cone of invariant weight vectors  $\C(C_n(\tau))$ which induce boundaries on $\tau$ which lie in $\PC(\tau)\subset \C(\tau)$.   These form the subspace $b_n\inverse(\PC(\tau))\subset \C(C_n)$.  Now $\C(C_n)$ is defined by equations and inequalities with integer coefficients, so the rational points (points with only rational entries) are dense.  To lie in $b_n\inverse(\PC(\tau))$, an invariant weight vector $\bov$ for $C$ must satisfy one more linear equation, with integer coefficients, in the entries $v_i$ of $\bov$, namely the equation which ensures that the sum of the induced weights on $\tau$ equals 1.  We conclude that rational points are dense in $b_n\inverse(\PC(\tau))$.

Now we choose rational point approximations $\bor_n$ of $\bov_n$ in $b_n\inverse(\PC(\tau))$ sufficiently close to ensure that $b_n(\bor_n)\to \bow$ and $\Chi(C_n(\bor_n))\to X(\bow)$.   This is possible because both $b_n$ and $\Chi$ are linear on $b_n\inverse(\PC(\tau))$.

Clearly $C_n(\bov_n)$ is a rational weighted surface; in other words, there exist integers $K_n$ so that $C_n(K_n\bov_n)=F_n$ is a surface without sphere components and $\bdry F_n$ is carried by $\tau$, say $\bdry F_n=\tau(\boy_n)$, where $\boy_n$ has integer entries.  Now $\Chi(F_n)/K_n=\Chi(C_n(\bov_n))\to X(\tau(\bow))$. Further, dividing $\boy_n$ by $K_n$ must yield $b_n(\bor_n)$, which is an element of $\PC(\tau)$.

Applying Lemma \ref{TechLemma}, we conclude that there is a $\tau$-select branched surface $B_i$ in the collection guaranteed by Lemma \ref{NormalLemma} and a measured lamination $B_i(\bov)$ satisfying $\Chi(B_i(\bov))\ge X(\bow)$.
\end{proof}

To help show that maximal-$\Chi$ Seifert laminations whose boundaries are carried by an oriented train track $\tau$ are neatly organized by branched surfaces, we prove the following lemmas.

\begin{lemma} \label{SumLemma} Suppose $\tau\embed \bdry M$ is an oriented train track.  Let $F_1,F_2,\ldots F_m$ be maximal-$\Chi$ surfaces embedded in $M$ with $\bdry F_i$ carried by $\tau$, $\bdry F_i=\tau (\boy_i) $, and let $\boy=\sum_i\boy_i$.   Then there exists a maximal-$\Chi$ surface $F$ with $\bdry F=\tau(\boy)$, satisfying $\Chi(F)\ge \sum_i \Chi(F_i)$.   
\end{lemma}

\begin{proof}  We use induction.  Consider the case $m=2$.  We have $\bdry F_1$ and $\bdry F_2$ carried by $\tau$, so we assume that $\bdry F_1$ and $\bdry F_2$ are embedded in $N(\tau)$ transverse to fibers, with $\bdry F_i=\tau(\boy_i)$, $i=1,2$.    We also assume $F_1$ and $F_2$ are transverse.   If $F_1$ and $F_2$ intersect on a closed curve which bounds a disk on one of the surfaces, we suppose without loss of generality that $\alpha$ is such a curve innermost on $F_1$, bounding a disk $E_1$ in $F_1$.   By the incompressibility of $F_2$, $\alpha$ also bounds a disk $E_2$ in $F_2$.   We can eliminate at least one curve of intersection between $F_1$ and $F_2$ by replacing $E_2\subset F_2$ by $E_1$ and isotoping a little more to push $F_2$ away from $E_2$.  After removing all closed curves of intersection trivial in $F_1$ or $F_2$, we perform oriented cut-and-paste on $F_1\cup F_2$ on remaining curves of intersection to obtain $G$.  Note that $\bdry G=\tau(\boy_1+\boy_2)$ and $\Chi(G)=\Chi(F_1)+\Chi(F_2)$.   The cut-and-paste operation cannot introduce spheres by construction; if $G$ contains other closed surfaces (with $\Chi\le 0$), we can discard them to obtain an oriented surface $F$ with the same boundary, and $\Chi(F)\ge \Chi(G)$.  If $F$ is not maximal-$\Chi$, we can replace it with a surface with $\Chi(F)$ even larger.  

Finishing the proof by induction is now straightforward.   Assume  we can replace $F_1\cup F_2,\cdots\cup F_{m-1}$ by a surface $G$ with $\Chi (G)\ge \sum_{i=1}^{m-1} \Chi(F_i)$, and with $\bdry G=\tau(\sum_{i=1}^{m-1}\boy_i)$.  Then we can apply the case $m=2$ to the two surfaces $G$ and $F_m$ to prove  our lemma for any $m$.
\end{proof}

\begin{lemma} \label{NonPos}  Suppose $(B,\bdry B) \embed (M,\bdry M)$ is an aspherical branched surface.   If $B(\bod)$ is a closed lamination, then $\Chi(B(\bod))\le 0$.
\end{lemma}

\begin{proof}   Let $\bou_0,\bou_1,\ldots,\bou_k$ be the vertices of the convex polytope $\PC(B)$.  Then $\bod$ is a linear combination of these vertices, $\bod=\sum t_i\bou_i$, where $t_i\ge 0$ and $t_i>0$ only if $B(\bou_i)$ is closed.   Each $\bou_i$ has only rational entries, so $B(\bou_i)$ represents a weighted closed surface, which cannot include weighted spheres.   Therefore, if $t_i>0$, $\Chi(B(\bou_i))\le 0$, and it follows that $\Chi(B(\bod_i))=\sum t_i\Chi(B(\bou_i))\le0$.
\end{proof}

\begin{lemma} \label{TautLemma} Let $B\embed M$ be an aspherical oriented branched surface, $\bdry B=\tau\embed \bdry M$.   Suppose $B(\bou)$ is a maximal-$\Chi$ Seifert lamination and $\bou$ has strictly positive entries.  Then any measured lamination carried by $B$ is a maximal-$\Chi$ Seifert lamination or a closed lamination of Euler characteristic 0.
\end{lemma}

\begin{proof}  The idea of this is easy, but the details are technical.    The idea:  if $B(\bou)$ is maximal-$\Chi$ and $B(\bou_0)$ is not maximal-$\Chi$, then there exists $ B'(\bou_0')$  with $\bdry B'(\bou_0')=\bdry B(\bou_0)$ with larger $\Chi$. Then  ``remove a multiple of $B(\bou_0)$ from $B(\bou)$ and replace it with the same multiple of $B'(\bou_0')$, to obtain a contradiction to the assumption that $B(\bou)$ has maximal $\Chi$."

 Let $\bou_0,\bou_1,\ldots,\bou_k$ be the vertices of the convex polytope $\PC(B)$.  These vertices must have all rational entries.  By scaling, we can assume that the sum of the entries of $\bou$ is 1, which means $\bou\in \PC(B)$.  Our first goal is to show that for each $i=0,1,2,\ldots, k$ if $B(\bou_i)$ is not a closed lamination, then it is maximal-$\Chi$.  (If $B(\bou_i)$ is closed, ``maximal-$\Chi$" is not defined.)  Suppose $B(\bou_0)$ is not closed and not maximal-$\Chi$.  Then there exists an aspherical branched surface $B'$ with $\bdry B'$ a sub- train track of $\tau$ and there exists an invariant weight vector $\bou_0'$ for $B'$ such that $\bdry \bou_0'=\bdry \bou_0$ and $\Chi(B'(\bou_0'))>\Chi(B( \bou_0))$.   Since $\Chi(B'(\cdot))$ is linear on $\C(B')$ we can assume all entries of $\bou_0'$ are rational.  Now $\bou$ can be expressed as a convex combination of the $\bou_i$'s, $\bou=\sum_{j=0}^kt_j\bou_j$.  Without loss of generality, $t_j>0$.   Approximating $t_j$'s by rational numbers, for each $j$ we can find a sequence $\{r_{nj}\}$ of rational numbers such that $r_{nj}\to t_j$, $\sum_{j=0}^kr_{nj}\bou_j\to \bou$ and $\sum_{j=0}^kr_{nj}=1$.  Then the linearity of the Euler characteristic function implies $\Chi(B(\sum_{j=0}^kr_{nj}\bou_j))\to \Chi(B(\bou))$.   For a fixed $n$, scaling all weight vectors by an appropriate (large) integer $L_n$ ensures that $L_nr_{nj}\bou_j$, $j=0,1,\ldots k$, and $L_nr_n \bou'_0$ all have only integer entries.  This means $B(\sum_{j=0}^kL_nr_{nj}\bou_j)$ is a surface $G_n$, while $B(L_nr_{nj}\bou_j)$ is a surface $G_{nj}$ and $B'(L_nr_{n0}\bou_0')$ is a surface $G_{n0}'$.  If we let $\bos_n=\sum_{j=0}^kL_nr_{nj}\bou_j$, then $G_n=B(\bos_n)$, and $\bos_n/L_n\in \PC(B)$.   This means normalizing the weight vector $\bos_n$ induced by $G_n$ on $B$ by dividing by the sum of its entries yields an element $\bos_n/L_n\in \PC(B)$.

Now we have $$\Chi(B(\bou))=\lim_{n\to \infty}\sum_{j=0}^k\Chi(B(r_{nj}\bou_j))=\lim_{n\to \infty}\sum_{j=0}^k\frac 1{L_n}\Chi(B(L_nr_{nj}\bou_j))=$$

$$=\lim_{n\to \infty}\sum_{j=0}^k\frac 1{L_n}\Chi(G_{nj})=\lim_{n\to \infty}\frac 1{L_n}\Chi(G_n).$$

\noindent  We will eventually arrive at a contradiction by replacing $G_{n0}$ in the above by $G_{n0}'$ as follows.   Since $\Chi(B'(\bou_0'))>\Chi(B( \bou_0))$, say $\Chi(B'(\bou_0'))-\Chi(B( \bou_0))=\epsilon>0$, we also have 
$$\Chi(G_{n0}')-\Chi(G_{n0})=\Chi(B'(L_nr_{n0}\bou_{0}'))-\Chi(B(L_nr_{n0}\bou_{0}))=$$
$$=L_nr_{n0}\left[\Chi(B'(\bou_{0}'))-\Chi(B(\bou_{0}))\right]=L_nr_{n0}\epsilon>0.$$
\noindent We apply Lemma \ref{SumLemma} to the surfaces $G_{n0}', G_{n1}, G_{n2},\ldots G_{nk}$.  Using the lemma, by suitable isotopy and oriented cut-and-paste of these surfaces (all of whose boundaries are carried by $\tau$) we obtain a surface $F_n$ with $\Chi(F_n)\ge \Chi(G_{n0}' )+\sum_{i=1}^k\Chi(G_{ni})$.   Then $\Chi(F_n)-\Chi(G_n) \ge \left[\Chi(G_{n0}' )+\sum_{i=1}^k\Chi(G_{ni})\right]-\sum_{i=0}^k\Chi(G_{ni})=\Chi(G_{n0}')-\Chi(G_{n0})=L_nr_{n0}\epsilon$.  The surface $F_n$ has the same boundary as $G_n$, which is fully carried by $\tau$.   

Now we will normalize the surfaces $G_n$ differently by dividing by the the sum $K_n$ of the entries of $\boy_n=\bdry \bos_n$.   After passing to a subsequence, we arrange $\boy_n/K_n=\bdry \bos_n/K_n\to \bow$ for some $\bow\in \PC(\tau)$.   On the other hand, $\bdry \bos_n/L_n\to \bdry \bou$.  Since $\bdry \bos_n/K_n$ and $\bdry \bos_n/L_n$ are in the same projective class, and neither sequence approaches the zero vector, their limits are in the same projective class, and there exists $\lambda>0$ so that $\bow=\lambda\bdry\bou$.  Now suppose $s_n$ is an entry of $\bos_n$ for a sector intersecting a segment of $\tau$ with a nonzero weight $w$ assigned by $\bow$ or $\lambda \bdry \bou$.  Then $\lim s_n/K_n\to  w$ while $ s_n/L_n\to w/\lambda$, from which we conclude $\lambda=\lim L_n/K_n$.

We know $\Chi(F_n)-\Chi(G_n)\ge L_nr_{n0}\epsilon$.   It follows that 
$$\frac{\Chi(F_n)}{L_n}-\frac{\Chi(G_n)}{L_n}\ge r_{n0}\epsilon$$
$$\frac{\Chi(F_n)}{L_n}\ge\frac{\Chi(G_n)}{L_n}+ r_{n0}\epsilon$$

\noindent Multiplying both sides by $L_n/K_n$, 
$$\frac{\Chi(F_n)}{K_n}\ge\frac{L_n}{K_n}\left(\frac{\Chi(G_n)}{L_n}\right)+\frac{L_n}{K_n}r_{n0}\epsilon$$
\noindent  Taking limit superior as $n\to \infty$ we get:
$$\limsup_{n\to \infty}\frac{\Chi(F_n)}{K_n}\ge\lambda \Chi(B(\bou))+\lambda t_0\epsilon$$

Next we apply Lemma \ref{TechLemma} to conclude there is a normal $\tau$-select branched surface $B_i$ and a measured lamination $B_i(\bov)$ with $\bdry \bov=\bow$ and with $\Chi(B_i(\bov))\ge\lambda \Chi(B(\bou))+\lambda t_0\epsilon$.  Then $\Chi(B_i(\bov/\lambda))\ge\Chi(B(\bou))+t_0\epsilon$.   This contradicts the assumption that $B(\bou)$ is a maximal-$\Chi$ Seifert lamination since $\bdry (\bov/\lambda)=\bdry\bou$.

We can use the same argument to prove that for any $\bou_0$ in $\PC(B)$ with all rational entries, $B(\bou_0)$ is maximal-$\Chi$.   In this case, we assume the extremal points of $\PC(B)$ are $\bou_1,\bou_2,\ldots,\bou_k$.    Then we can write $\bou$ as a convex combination of $\bou_0,\bou_1,\bou_2,\ldots,\bou_k$ as before, $\bou=\sum_{i=0}^kt_i\bou_i$ with $t_0>0$, and proceed as before.  

Next, we must show that if $\bod$ is any point in $\PC(B)$ and $B(\bod)$ is not closed, then $B(\bod)$ is maximal-$\Chi$.  Let $\hat B$ be the sub- branched surface of $B$ which fully carries $B(\bod)$, and let $B(\bod)=\hat B(\hat\bod)$ where $\hat \bod$ is the restriction of $\bod$ to $\hat B$.   Suppose $B(\bod)$ is not maximal-$\Chi$.   Then there exists an aspherical branched surface $C\embed M$  and an invariant weight vector $\boh$ (positive entries) for $C$ such that  $\bdry C(\boh)=\bdry B(\bod)$ as a lamination, with $\Chi(C(\boh))>\Chi(B(\bod))=\Chi(\hat B(\hat\bod))$.   This means that by splitting $C(\boh)$ and $\hat B(\hat\bod)$ along leaves, we can replace $C$, $\hat B$, $\boh$ and $\hat \bod$ such that $\bdry C=\bdry\hat B$.  The new branched surfaces $C$ and $\hat B$ remain aspherical.  Identifying $\bdry C\subset C$ with $\bdry \hat B\subset \hat B$ we obtain a closed branched surface $A$.   In fact, we can embed $A$ in the double of the manifold $M$.   Also, we have an invariant weight vector $\boa$ for $A$ whose restriction to $\hat B$ and $C$ is $\hat \bod$ and $\boh$ respectively.  We choose a sequence of weight vectors $\boa_n$ for $A$, each with only rational entries, such that $\boa_n\to \boa$.   By restricting to $\hat B$ and $C$ we get sequences $\bob_n$ and $\boc_n$, each with only rational entries, and with $\bob_n\to\hat \bod$, $\boc_n\to \boh$.  Since $\lim_{n\to\infty}\Chi(C(\boc_n))=\Chi(C(\boh))>\Chi(B(\bod))=\lim_{n\to\infty}\Chi(\hat B(\bob_n))$, we must have $\Chi(C(\boc_N))>\Chi(\hat B(\bob_N))$ for some sufficiently large $N$.  On the other hand, since $\bdry C(\boc_N)=\bdry \hat B(\bob_N)$, it follows that $\hat B(\bob_N)$ is not maximal-$\Chi$.   This contradicts what we have already proved:  that for any rational $\bod_0$, $B(\bod_0)$ is maximal-$\Chi$.

Finally, we must show that if $B(\bod)$ is closed, then $\Chi(B(\bod))=0$.   Since $B$ is aspherical $\Chi(B(\bod))\le 0$, by Lemma \ref{NonPos}.  If $\Chi(\bod)< 0$, then for a sufficiently small $\delta$, $\bou-\delta\bod$ has positive entries, and $\Chi(B(\bou-\delta\bod))=\Chi(B(\bou))-\delta(\Chi(B(\bod))>\Chi(B(\bou))$, whereas $\bdry(\bou-\delta\bod)=\bdry\bou$.  This contradicts the assumption that $B(\bou)$ is a maximal-$\Chi$ Seifert lamination:  $B(\bou-\delta\bod)$ is a Seifert lamination with the same boundary as $B(\bou)$ with larger $\Chi$.
\end{proof}

\begin{defn}  \label{TautBranched} Let $B\embed M$ an oriented branched surface with $\bdry B$ a sub- train track of $\tau$.   Then $B$ is {\it $\tau$-taut} or {\it taut} if it is a sub- branched surface of a $\tau$-select  branched surface and fully carries a maximal-$\Chi$ Seifert lamination.  A Seifert lamination is {\it taut} if it is carried by a taut branched surface.  We do not require $M$ to be irreducible in this definition. \end{defn}

\begin{remark}  Our definition says a Seifert surface, with boundary carried by $\tau$, is taut if it is maximal-$\Chi$ as a lamination.   Since it is then also maximal-$\Chi$ as a surface, it is fully carried by a $\tau$-select branched surface, so it is taut when viewed as a Seifert lamination.  
\end{remark}

Now we can prove a stronger version of Lemma \ref{NormalLemma}:

\begin{thm} \label{TautThm} Suppose $M$ is a compact, orientable, irreducible 3-manifold with boundary and $\tau\embed \bdry M$ is an oriented train track in $\bdry M$.  Then there exists a finite collection $\{T_i:  i=1,2,\ldots m\}$ of $\tau$-taut oriented branched surfaces such that for any peripheral lamination link $\tau(\bow)$ which bounds a Seifert lamination, there exists a $T_i$ which fully carries a taut Seifert lamination for the link.   The leaves of a taut Seifert lamination carried by $T_i$ are $\pi_1$-injective.
\end{thm}

\begin{proof}  Given a peripheral lamination link $\tau(\bow)$, by Theorem \ref{AchieveX} there is a $\tau$-select branched surface $B_i$ which carries a maximal-$\Chi$ Seifert lamination $B_i(\bov)$ for $\tau(\bow)$.   It may not be fully carried by $B_i$, but it is fully carried by some sub- branched surface $T$ of $B_i$, say $B_i(\bov)=T(\bou)$.  For fixed $i$ there are just finitely many sub-branched surfaces of $B_i$, so there are only finitely many possibilities for $T$.   Considering the finite set of all $B_i$, there are still only finitely many possibilities for $T$.   Then by Lemma \ref{TautLemma}, every Seifert lamination carried by $T$ is maximal-$\Chi$.   By definition $T$ is taut.   The set $\{T_i\}$ is the finite set of all sub-branched surfaces of $B_i$'s which fully carry maximal-$\Chi$ Seifert laminations.  The leaves of a Seifert lamination carried by one of the taut branched surfaces $T_i$ are $\pi_1$-injective because the Seifert lamination is carried by $\tau$-select branched surface, since $T_i$ is a sub- branched surface of a $\tau$-select branched surface.
\end{proof}

Probably the taut branched surfaces $T_i$ in the above statement can be modified so they are $\tau$-select.   The main difficulty is eliminating disks of contact.
 
 We use the following proposition to deal with peripheral links in reducible 3-manifolds.
 
 \begin{proposition} \label{SumProp}  Suppose the compact, orientable 3-manifold $M$ has a prime decomposition $M=M_1\#M_2\#\cdots \#M_k$.  Suppose that for each $i$ , $\tau_i(\bow_i)$ is a possibly empty peripheral lamination link in $\bdry M_i$.   Let $\tau=\cup_i\tau_i\subset \bdry M$ with invariant weight vector $\bow$ assigning the same weights as the $\bow_i$'s.   Then $\tau(\bow)$ bounds a Seifert lamination in $M$ if and only if $\tau_i(\bow_i)$ bounds a 2-dimensional measured lamination in $M_i$ for each $i$ .  Further, $X(\tau(\bow))=\sum_i X(\tau_i(\bow_i))$ and there is a taut branched surface $B \embed M$ and a taut Seifert lamination  $B(\bov)$ such that $\bdry B(\bov)=\tau(\bow)$.  ($X(\tau_i(\bow_i))$ refers to Seifert laminations in $M_i$.)
 
 \hip
 \noindent There is a finite collection $\{T_i\}$ of taut branched surfaces in $M$ such that if $\tau(\bow)\subset \bdry M$ is a peripheral lamination link which bounds, then it bounds a taut Seifert lamination fully carried by one of the  $T_i$.
 \end{proposition}
 
 \begin{proof}  Let $S=S_1\cup S_2\cup\cdots \cup S_m$ be a union of essential spheres in $M$, such that cutting $M$ on $S$ and yields a collection of holed irreducible 3-manifolds, $M_1', M_2',\ldots, M_k'$.   Capping the boundary spheres yields  $M_1, M_2,\ldots, M_k$.
 
 One implication is easy, namely if $\tau_i(\bow_i)$ bounds $(\Lambda_i,\mu_i)$ in $M_i$ for each $i$, then $\tau(\bow)$ bounds $\cup_i \Lambda_i$ in $M$, since we can assume $\Lambda_i$ lies in $M_i'\subset M$.  This also shows $X(\tau(\bow))\ge \sum_i X(\tau(\bow_i)$.
 
For the converse, we will show that if $\tau(\bow)$ bounds a Seifert lamination $(\Lambda,\nu)=B(\bov)$ in $M$, then we can replace it with a Seifert lamination $(\Lambda',\nu')$ disjoint from $S$, and also disjoint from all $M_i'$ corresponding to closed summands $M_i$.   We isotope $B$ and $\Lambda$ so it is transverse to $S$.   Then, because $\Lambda$ is measured, the intersection $\Lambda\cap S_i$ with one of the sphere components of $S$ is a finite collection of measured families of isotopic closed curves.  We choose an innermost family on $S_i$, and we surger the lamination just as we would surger a surface.  By splitting $B(\bov)$ appropriately along leaves, we can modify $B(\bov)$ at the same time.   Repeating this process a finite number of times yields a measured lamination $(\Lambda',\nu')=B'(\bov')$ disjoint from $S$, and we let $(\Lambda_i,\nu_i)=B_i(\bov_i)$ be the intersection of $(\Lambda',\nu')=B'(\bov')$ with $M_i'$, which can also be regarded as a lamination in $M_i$.   If $M_i$ is closed, then $\Lambda_i$ is closed and we can discard it.

Next we must show  (assuming $\tau(\bow)$ bounds) that $X(\tau(\bow))\le \sum_i X(\tau(\bow_i))$.  Suppose not.  Then $X(\tau(\bow))>\sum_iX(\tau_i(\bow_i))$.   By the continuity of $X$, there exists an invariant weight vector $\bor$ for $\tau$ with only rational entries such that $X(\tau(\bor))>\sum_iX(\tau_i(\bor_i))$, where $\bor_i$ is the restriction of $\bor$ to $\tau_i$.  Suppose $X(\tau(\bor))=\sum_iX(\tau_i(\bor_i))+\epsilon$, where $\epsilon>0$.   Then we can find a Seifert lamination $B(\bou)$ for $\tau(\bor)$, where $B$ is aspherical and $\bou$ has only rational entries, $\bdry B(\bou)=\tau(\bor)$, and  $\Chi(B(\bou))>X(\tau(\bor))+\epsilon/2$.  So $B(\bou)$ is a weighted surface which we will denote $qF$, where $q$ is a rational weight on $F$.  $F$ has no sphere components, and we can discard all other closed components of $F$ without decreasing $\Chi(F)$.  Now we make $F$ transverse to $S$.   Consider a closed curve of $F\cap S$ innermost on $S$, bounding a disk $H$ in $S$.   Surgering $F$ using $H$ yields a new oriented surface $F$ intersecting $S$ in fewer curves.   If surgery produces a sphere component, we discard it, and $\Chi(F)$ is unchanged.     Otherwise $\Chi(F)$ increases.  Repeating surgeries to eliminate intersections with $S$, we end with a new $F$ disjoint from $S$, and with $\Chi (F)$ no smaller than for the original $F$.     We can write the modified $F$ as $F=\cup_i F_i$ where $F_i\subset M_i'$.   For this $F$, we have $\Chi(qF)>X(\tau(\bor))+\epsilon/2$.   Hence $\sum_i\Chi(qF_i)>X(\tau(\bor))+\epsilon/2$.   But $\Chi(qF_i)\le X(\tau(\bor_i))$, because $qF_i$ is a Seifert lamination for $\tau(\bor_i)$ in $M_i$, carried by an aspherical branched surface (namely $F_i$).   This implies  $X(\tau(\bor))+\epsilon/2<\sum_i\Chi(qF_i)\le \sum_iX(\tau(\bor_i))$, which contradicts our assumption that  $X(\tau(\bor))>\sum_iX(\tau_i(\bor_i))$.

It remains to show that there exist finitely many taut branched surfaces $\{T_i\}$ in $M$ such that for every peripheral lamination $\tau(\bow)$ which bounds, there is a taut Seifert lamination $T_i$ which fully carries a taut Seifert lamination for $\tau(\bow)$.  Since any taut Seifert lamination can be approximated by a taut weighted surface, it is enough to show that there exist finitely many taut $T_i$ such that every taut Seifert surface is carried by some $T_i$.   Given a taut Seifert surface $F$, we can modify $F$ to eliminate intersections with $S$, as above.  Then clearly the new $F$ is a disjoint union of taut Seifert surfaces $F_j$ in $M_j'$, which can also be regarded as taut Seifert surfaces in $M_j$.   Thus $F_j$ is fully carried by one of finitely many taut branched surfaces $T_{ji}$ in $M_j$.  Then $F$ is carried by some taut branched surface $T$ in $M$ such that for each $j$, $T\cap M_j'$ is $T_{ji}$ for some $i$.   Thus there are finitely many possibilities for $T$.
\end{proof}

Proposition \ref{SumProp} makes it easier to prove facts about the function $X$ on $\CB(\tau)$.   Since $X$ is continuous, proving an equation or inequality involving $X$ for rational points (or weighted surfaces) is sufficient to prove it for all points of  $\CB(\tau)$ using the following corollary, which could be regarded as a stronger version of Lemma \ref{TechLemma}.

\begin{corollary} \label{ApproxCor}  Suppose $M$ is a compact, oriented 3-manifold.   If $\tau\embed \bdry M$ is an oriented train track and $\tau(\bou)$ represents a peripheral lamination link which bounds a Seifert lamination, then $X(\tau(\bou))=P$ if and only if there exists a sequence of taut weighted properly embedded surfaces $q_nF_n$ with $\bdry F_n$ carried by $\tau$,  $q_n\bdry F_n=\tau(\bou_n)$, $\bou_n\to \bou$, and $q_n\Chi(F_n)\to P$.  The weights $q_n$ can be chosen to be rational.
\end{corollary}

\begin{proof}  Suppose $X(\tau(\bou))=P$.   Then there exists a taut branched surface $T$ which fully carries a taut Seifert lamination $T(\bov)$ for $\tau(\bou)$.   Choosing a sequence of rational invariant weight vectors $\bov_n$ for $T$ approaching $\bov$, we can interpret $T(\bov_n)$ as a weighted surface $q_nF_n$.   $F_n$ must be taut, because it is carried by a $\tau$-taut branched surface.  By the linearity of $\Chi$ on $\C(T)$, $\Chi(T(\bov_n))\to \Chi(T(\bov))$, hence $\Chi(q_nF_n)=q_n\Chi(F_n)\to \Chi(T(\bov))=X(\tau(\bou))$.

Now suppose there exists a sequence of weighted properly embedded taut surfaces $q_nF_n$ in $M$ with $q_n\bdry F_n=\tau(\bou_n)$, $\bou_n\to \bou$, and $q_n\Chi(F_n)\to P$.  
We will apply Lemma \ref{TechLemma} to the sequence $\{F_n\}$.   Let $\bdry F_n=\tau(\boy_n)$ and let $K_n$ be the sum of entries of $\boy_n$.  Then $\tau(\bou_n/q_n)=\bdry F_n=\tau(\boy_n)$, $\boy_n=\bou_n/q_n$.  Let $\bow_n=\boy_n/K_n=\bou_n/(K_nq_n)$.  Then $\bow_n=\boy_n/K_n$ must converge to a vector in the projective class of $\bou$, say $\lambda\bou$ for some $\lambda>0$.   This shows $1/(K_nq_n)\to \lambda$.  Also $\Chi(F_n)/K_n=\Chi(q_nF_n)/(q_nK_n)\to \lambda P$.  By Lemma \ref{TechLemma}, there is a $\tau$-select branched surface $B$ and a measured lamination $B(\bov)$ such that $\bdry B(\bov)=\tau( \lambda\bou)$ and $\Chi(B(\bov))=\lambda P$.
Then  $\Chi(B(\bov/\lambda))=P$ and $\bdry B(\bov/\lambda)=\tau(\bou)$, which proves $X(\tau(\bou))\ge P$.  Since $B$ fully carries a taut surface, $B(\bov/\lambda)$ is maximal-$\Chi$, so $X(\tau(\bou))= P$.
\end{proof}

Most of the work has been done for proving the theorems stated in the introduction about peripheral laminations.   

\begin{proof}[Proof of Theorem \ref{XThm}]  The fact that the function $X:\CB(\tau)\to \reals$ is linear on rays of $\CB(\tau)$ is obvious.  

We show that the function $X$ is finite piecewise linear on $\CB(\tau)$.   For every taut $T_i$ described in Proposition \ref{SumProp}, the function $\bdry:\C(T_i)\to \CB(\tau)$ is linear, replacing the entries of $\bov$ by a subset of the entries, with some repetitions.    Now $\Chi(T_i(\bov))$ is a linear function of $\bov$, and $\Chi(T_i(\bov))=X(\bdry \bov)$.  If $\bdry \bov_1=\bdry \bov_2$, then $\Chi(T_i(\bov_1))=\Chi(T_i( \bov_2))=X(\bdry\bov_1)$.   Thus, viewing $\C(T_i)\subset \reals^m$ and $\CB(\tau)\subset \reals^n$, we can regard $\bdry$ as a linear map $\bdry:\reals^m\to \reals^n$, and $\Chi$ is constant on cosets of $\ker(\bdry)$ intersected with $\C(T_i)$.   Therefore $X$ is linear on $\im(\bdry)\subset \CB(\tau)$.  

We know that $\CB(\tau)$ is covered by the finite number of images of the linear boundary maps $\bdry:\C(T_i)\to \CB(\tau)$, hence $X$ is piecewise linear.

To prove concavity of the function $X:\CB(\tau)\to \reals$, we must show $X(t\bow_0+(1-t)\bow_1)\ge tX(\bow_0)+(1-t)X(\bow_1)$.  As usual, we prove this by approximating $t$ by a rational number and approximating $\bow_i$ by a vector with rational entries and clearing denominators so we can work with surfaces.  We will start with the surfaces and work backwards.   Given curve systems $\tau(\boy_0)$ and $\tau(\boy_1)$, where each $\boy_i\in \CB(\tau)$ has integer entries, and given a positive integer $j$,  $0\le j\le k$, suppose $\tau(\boy_0)$ bounds a taut surface $G$ and $\tau(\boy_1)$ bounds a taut surface $H$ (carried by one of the $T_i$).  Now let $F_1,F_2, \ldots F_j$ all be surfaces isotopic to $G$, and let  $F_{j+1},F_{j+2}, \ldots F_k$ all be surfaces isotopic to $H$.  Then by Lemma \ref{SumLemma}, there is a surface $F$ with $\bdry F=\tau(j\boy_0 +(k-j)\boy_1)$ and $\Chi(F)\ge\sum_{i=1}^k\Chi( F_i)=jX(\boy_0)+(k-j)X(\boy_1)$.  This implies $X(j\boy_0 +(k-j)\boy_1)\ge jX(\boy_0)+(k-j)X(\boy_1)$.  Dividing both sides of the equation by $k$, we have

 $$\displaystyle X\left(\frac jk\boy_0 +\frac{k-j}k\boy_1\right)\ge \frac jk X(\boy_0)+\frac{k-j}kX(\boy_1).$$  

\noindent Letting $t=j/k$, we have $X(t\boy_0 +(1-t)\boy_1)\ge tX(\boy_0)+(1-t)X(\boy_1)$.  This is the inequality we want, except here $t$ is rational and the weight vectors $\boy_0$ and $\boy_1$ have integer entries.  We can change $\boy_i$ to  $\bow_i$ with rational entries by dividing both sides of the inequality by another integer, $m$ say, to  get  $$\displaystyle X\left(t\frac{\boy_0}m +(1-t)\frac{\boy_1}m\right) \ge tX\left(\frac{\boy_0}m\right)+(1-t)X\left(\frac{\boy_1}m\right).$$

\noindent  In fact, in this way we can get arbitrary weight vectors $\bow_0$ and $\bow_1$ with rational entries,  as $\displaystyle \bow_0=\frac{\boy_0}m$ and $\displaystyle \bow_1=\frac{\boy_1}m$.   Now we can easily prove the general case by approximation, using Corollary \ref{ApproxCor} .
\end{proof}

%Then by Lemma \ref{TechLemma} we pass to a subsequence so that for all $n$, $F_n$ is carried by the same $\tau$-select branched surface $B_i$.    Then $\bou_n=q_n\boy_n$ and $\bow_n=\boy_n/K_n$ must converge to vectors of $\CB(\tau)$ in the same projective class, so $\bou=\lambda\bow$ for some $\lambda>0$.  Then    By Lemma \ref{TechLemma}  Since $F_n$ is taut, it is fully carried by one of the taut branched surfaces $T_i$ guaranteed by Proposition \ref{SumProp}.   Passing to a subsequence, we can assume that for all $n$, $F_n$ is carried by the same $\tau$-taut branched surface $T$, say $q_nF_n=T(\bov_n)$, where $\bdry\bov_n=\bow_n$.   

%Without loss of generality, we assume $\bou\in \PCB(\tau)$.  By passing to a subsequence, we may assume that for all $n$, $F_n$ is fully carried by some fixed $\tau$-select branched surface $B_i$, $q_nF_n=B_i(\bov_n)$.   Passing to a subsequence again

\begin{proof}[Proof of Theorem \ref{SeifertThm}]  Theorem 1.8, stated in the introduction easily follows from the results we have proved.
\end{proof}

\section{From peripheral to arbitrary links.}\label{FromTo}

Theorem \ref{Extend} could be regarded as a corollary of Theorem \ref{AchieveX}.  It extends the results of the previous section to non-peripheral lamination links in compact, reducible, oriented 3-manifolds $M$.  

The definition of maximal-$\Chi$ Seifert laminations for non-peripheral links is the same as for peripheral links, and so is the definition of $X$ as a function on the set of links, see Definition  \ref{KnottedDef1}.

Suppose $V_\bow(\tau)$ represents a (non-peripheral) lamination link in $M$.  The added subtlety in the non-peripheral case comes from the fact that if we represent a link by a fixed prelamination $V_\bow(\tau)$, a (maximal-$\Chi$) Seifert lamination may not be representable as $B(\bov)$ with $\bdry B=\tau$.  We can say this in another way:  if $V_\bow(\tau)$ represents a lamination $(L,\mu)$, which we regard as a lamination $L$ in the surface $V(\tau)$, then the leaves of a maximal-$\Chi$ Seifert lamination $(\Lambda,\nu)$ for $(L,\mu)$ may intersect $V(\tau)\setminus L$ in an essential way.   This means that we must choose our representation of the link $(L,\mu)$ as $V_\bow(\tau)$ more carefully.   It turns out that one must choose an irreducible, adigonal, anannular representative $V_\bow(\tau)$.  To show this is possible, we will first study scalloped surfaces.   A fibered neighborhood $V(\tau)$ is related to an underlying scalloped surface $S$ as shown in Figure \ref{ScallopedS}, by ignoring the fibering of $V(\tau)$ by intervals.

\begin{figure}[ht]
\centering
\scalebox{0.5}{\includegraphics{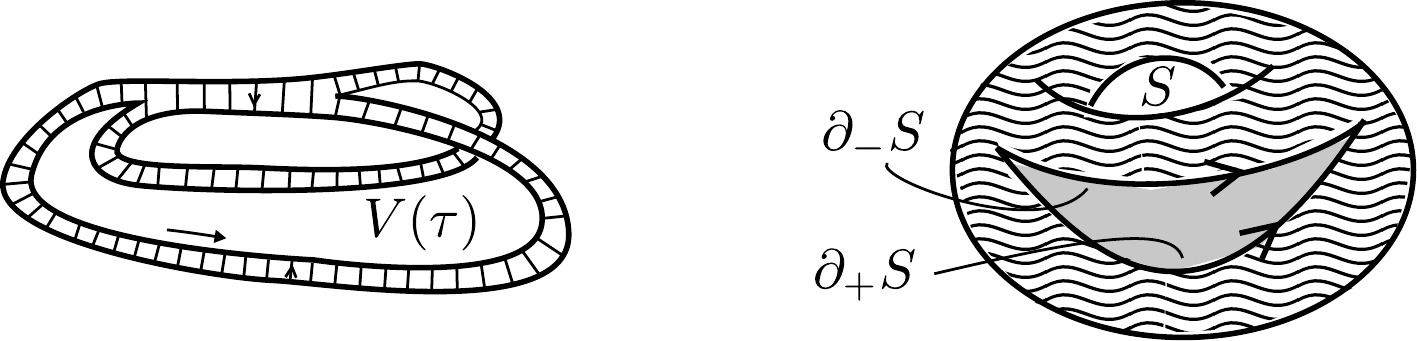}}
\caption{\footnotesize  The underlying scalloped surface associated to $V(\tau)$}
\label{ScallopedS}
\end{figure}

\begin{defn}  A {\it scalloped surface} $S$ is an oriented surface with inward cusps  on its boundary and with $\bdry S$ divided into finitely many maximal smooth curves with consistent tangential orientations.   The orientation of $S$ together with the tangential orientation of smooth curves of $\bdry S$ induce transverse orientations on the smooth boundary curves. The set of points where transverse orientations point to the interior of $S$ is called $\bdry_-S$, the set of points where the transverse orientation is outward is called $\bdry_+S$.   On components of $\bdry S$ subdivided into smooth arcs, the arcs have alternating transverse orientations but consistent tangential orientations at cusps.  We let $C$ denote the union of cusps.  An {\it oriented measured foliation} $\F$ of $S$ is an oriented transversely measured foliation of $S\setminus C$ tangent to $\bdry S$, whose orientation at $\bdry S$ is consistent with the tangential orientation of $\bdry S$, extended smoothly to cusp points.   
%More generally, if $S$ is not connected, and $\hat S$ is a non-empty union of components of $S$, then we also say that an oriented measured foliation $\F$ of $\hat S$ is an {\it oriented measured foliation} of $S$.  For a connected scalloped surface, an oriented measured foliation is always non-trivial and foliates all of $S$.
  A {\it preleaf} of $\F$ is a leaf of  $\F\setminus C$ completed by any cusp points at one or both ends of the leaf of $\F\setminus C$.   In addition, the foliation is equipped with a (locally finite, full support) transverse measure $\mu$. Thus the oriented measured foliation is actually a pair $(\F,\mu)$.  A {\it leaf} of $\F$ is defined as in the introduction for $V_\bow(\tau)$, see Definition \ref{LeafDef}. \end{defn}

We will often write ``foliation" to mean ``oriented measured foliation." 
Observe that the orientation on $S$ and a tangential orientation on $\F$ yields a transverse orientation for $\F$.   
Conversely, the orientation of $S$ and a transverse orientation for $\F$ yields a tangential orientation for $\F$.

It is often enough to understand measured oriented foliations of connected scalloped surfaces.

\begin{defn}  If $\F$ is a foliation of a connected scalloped surface $S$, and $\ell$ is a preleaf which has one end limiting at $C_0\in C$ (or two ends limiting on $C_0$ and $C_1$) then $\ell\cup\{ C_0\}$  (or $\ell \cup \{C_0,C_1\}$) is a  {\it separatrix}.   If $\gamma$ is a compact arc in a preleaf, disjoint from $\bdry S$ except possibly at endpoints,  then we can {\it split} (cut) $S$ on $\gamma$ to get a new scalloped surface with a new foliation.    A {\it regular splitting of $S$ on leaves of $\F$} is a splitting on finitely many compact arcs in separatrices, where each arc has exactly one end at a cusp.   A {\it regular pinching} is the inverse of a regular splitting.   Two foliations are {\it equivalent by regular pinching and splitting} if one can be obtained from the other by a finite number of regular pinching and splitting operations.

Associated to a foliation $\F$ of $S$ there is a possibly non-compact train track immersed in $S$, which is the union of separatrices and $\bdry S$.   We call this the {\it splitting train track} for $\F$.   We use the same terminology if $\F=V_\bow(\tau)$; the  {\it splitting train track} for $V_\bow(\tau)$ is the union of separatrices and the boundary of the surface $V_\bow(\tau)$.

If $\F$ has a compact separatrix $\gamma$ (both ends at cusps), then we say $\gamma$ is a {\it reducing arc} and we say $\F$ is {\it reducible}.  If $S$ contains a torus $\hat S$ and  $\hat S$ is foliated by circles, then any of these circle leaves is a {\it reducing curve} and $\F$ is {\it reducible}.   If $\F$ has no reducing arcs  or closed curves, we say $\F$ is {\it irreducible}.
\end{defn}

Notice that a leaf of $\F$ may be a preleaf, or it may appear in $\F$ as an embedded curve in the splitting train track.

Since a reducible oriented foliation of a scalloped surface $S$ foliates a simpler scalloped surface obtained by splitting on reducing arcs and/or closed curves, it is natural to assign the foliation to the simpler scalloped surface.    

Note that if we split on an arc $\gamma$ in a separatrix  with exactly one point of $\bdry\gamma$ at a cusp, then we obtain a new scalloped surface which is isomorphic to the old one.  Thus regular splitting on a leaf of $\F$ always yields a scalloped surface diffeomorphic to the original $S$.    If we split on a reducing arc $\gamma$, then we obtain a new scalloped surface which has fewer cusps than the old one.  This is why we call $\gamma$ a reducing arc.  If we split on an arc $\gamma$ in a separatrix  with neither point of $\bdry\gamma$ at a cusp, then we increase the number of cusps on the boundary of the scalloped surface.   We shall avoid this last type of splitting.

An important observation about splitting tells us that regular splitting of a measured foliation $\F$ of $S$ yields a diffeomorphic scalloped surface $S'$ with a foliation $\F'$ of $S'$ isotopic to $\F$ when $S'$ is identified with $S$ using a suitable diffeomorphism:

\begin{proposition}  Suppose $S$ is a connected scalloped surface with an oriented measured foliation $\F$.   Suppose $\gamma$ is an arc in a separatrix with exactly one end at a cusp $C_0$.   Then splitting along $\gamma$ yields a diffeomorphic scalloped surface $S'$ with an oriented measured foliation $\F'$.   Let $E$ be a disk neighborhood of $\gamma$ in $S$ such that $E\cap \bdry S$ is an arc containing $C_0$ in its interior.  There is a diffeomorphism $\phi:S\to S'$, which is the identity in the complement of $E$ such that  $\phi\inverse(\F')$ is isotopic to $\F$.

Essentially the same statement can be made for any regular splitting involving finitely many disjoint compact arcs, each with exactly one end at a cusp.   
\end{proposition}

The proof should be clear.   From the proposition we conclude that regular splitting of $(S,\F)$ can be obtained by an isotopy of $\F$ in $S$.

\begin{lemma} \label{ReducingLemma} (a)  Suppose $\F$ is an oriented measured foliation of a scalloped surface $S$.   Then by splitting $S$ on all reducing curves, one obtains an irreducible oriented foliation $\hat \F$ on a scalloped surface $\hat S$.   In each component $S_i$ of $\hat S$, the foliation $\hat \F$ restricts to an irreducible $\F_i$.  

\noindent (b) An irreducible $\F_i$ on $S_i$ is either a product foliation by $S^1$'s of an annulus, an irrational slope foliation of a torus, or it has the property that the union of separatrices is dense.  Further, if an irreducible $(S_i,\F_i)$ is not a foliated annulus, every boundary component of $S_i$ contains at least two inward cusps.
\end{lemma}

\begin{proof}   (a) If there is a reducing arc in $\F$, joining two different cusps on $\bdry S$, we split on it to obtain a scalloped surface with fewer cusps.  After splitting on all compact separatrices joining different cusps, we have a scalloped surface $\hat S$ with components $S_i$, with oriented measured foliation $(\hat F,\hat \mu)$ having the property that every separatrix is non-compact.   Next, for any torus component which is foliated by circles, we split on one of the closed curve leaves.    In all cases, the result is a foliated scalloped surface $(\hat S,\hat \F,\hat \mu)$ which is irreducible.

(b) First we deal with a torus $S_i$.   If $\F_i$ is a foliation by circles, then it is reducible.  The only other measured foliations of a torus are irrational slope foliations, which are irreducible.  

Next we analyze foliations which contain closed leaves.   Let $S_i$ be a component of $\hat S$ foliated by $\F_i$, and suppose $\F_i$ contains a closed leaf.   If $\bdry S_i$ has a closed boundary component $K$ without cusps, say $K\subset \bdry_- S$, we claim there is a maximal measured product $K\times [0,m]$ in $\F_i$ with $K\times \{0\}=K$.  To see that there is some foliated product neighborhood of $K$, observe that if a segment of a leaf in a regular neighborhood of $K$ intersects twice a small transversal with one end at a point of $K$ without forming a closed leaf isotopic to $K$, then the foliation is not measured.  To show there is a maximal such product, observe that if $K_t=K\times \{t\}$ is a leaf in a product with $K=K_0$, and $K_t$ does not intersect $\bdry S_i$, then we can again argue that nearby leaves must be isotopic to $K_t$.  On the other hand, if a maximal product has the form $K\times [0,m)$ for some real $m$, then as $t\to m$, $K_t$ converges uniformly to a closed leaf $K_m$, which must intersect $\bdry S_i$.  $K_m$ cannot be contained in the interior of $S_i$, otherwise the product could be extended further.   If $K_m\nsubset \bdry S_i$, then $\hat \F$ has a reducing arc, a contradiction, so the component $S_i$ is a foliated product.    More generally, if $\F_i$ contains any closed leaf $K$, a maximal measured product $K\times [0,m] $  with $m>0$ (or $K\times [-m,0]$ or $K\times [-m_1,m_2]$, $m_i>0$) and $K\times 0=K$ gives a contradiction in the same way.  (Note that $K$ may intersect $\intr(\bdry_- S)$ or $\intr(\bdry_+S)$.)  We have shown that the foliation $\F_i$ contains a closed leaf only if $S_i$ is an annulus with a product foliation $\F_i$.
We must also consider the possibility that the maximal product takes the form $K\times S^1$, in case $S$ has torus components, but this is ruled out by irreducibility.  Therefore we can now assume every boundary component of $S_i$ includes at least two cusps and we can assume $\F_i$ contains no closed leaves.

Now we consider an $(S_i,\F_i,\mu_i)$ which is not an annulus and not a torus and contains no closed leaves.  Assuming $\F_i$ is irreducible, let  $\iota:\rho\to \F_i$ be the immersion of the splitting train track.  The separatrices are all non-compact.   So $S_i\setminus \cl(\iota( \rho))$ is a foliated open submanifold of $S_i$, whose completion $\hat S_i$ must be a foliated surface with infinite outward cusps (like the complement of a lamination in a closed surface).    Now $\bdry \hat S_i$ cannot contain a closed curve, since this would be a closed leaf in the original foliation, and $S_i$ would be an annulus.   The only other possibility for a component of $\hat S_i$ is a foliated infinite product $\reals\times [0,m]$ with two infinite cusps.   If this foliated product has non-trivial transverse measure, one can easily show that the transverse measure on $\F$ is not locally finite.  The contradiction shows that $\iota(\rho)$ is dense, and also that the union of separatrices is dense.
\end{proof}

Our next goal is to prove Proposition \ref{AdigonalProp}.    The statement indicates that a representative $V_\bow(\tau)$ of a lamination link in $M$ can be replaced by a new representative which is irreducible, adigonal, and anannular, and which is unique up to isotopy when viewed as an oriented measured foliation of a scalloped surface embedded in $M$.  We prepare for the proof with the following lemma:

\begin{lemma} \label{DigonPinch} Suppose $S\embed M$ is a scalloped surface.
\hip

\noindent   (i)  Suppose also that $D$, $E$ are complementary digons with $\bdry D=\bdry E$.   Suppose the component of $S\cup D$ containing $D$ is not a closed surface.   Then $S\cup D$ is isotopic in $M$ to $S\cup E$.
\hip\noindent
More generally, if $\bdry S$ has several boundary components bounding digons $D_1, D_2, ..., D_k$, and $S\cup D_1\cup\cdots \cup D_k$ is an embedded surface with no closed components, then up to isotopy, replacing each $D_i$ by another digon $D_i'$ with the same boundary, where digons $D_i'$  are disjoint, one obtains a surface $S\cup D'_1\cup\cdots \cup D'_k$ isotopic in $M$ to $S\cup D_1\cup\cdots \cup D_k$.
\hip
\noindent (ii)  Suppose $S$ has complementary annuli $A_1,A_2,\ldots, A_k$ such that $S\cup A_1\cup A_2\cup\cdots\cup A_k$ is an embedded surface with no closed components.   Suppose $\bdry A_i$ intersects $S$ only in annulus components of $S$.   Then replacing each $A_i$ with another complementary annulus $A_i'$ having the same boundary, such that the $A_i'$'s are disjoint, yields a scalloped surface $S\cup A_1'\cup A_2'\cup\cdots\cup A_k'$ isotopic to  $S\cup A_1\cup A_2\cup\cdots\cup A_k$.
\hip
\noindent (iii) Suppose $S$ has two complementary digons $D_1$ and $D_2$ such that $S\cup D_1\cup D_2$ is an embedded surface possibly including a component containing $D_1$ and $D_2$.  We are interested in the case where $S\cup D_1\cup D_2$ has a torus component containing $D_1$ and $D_2$.   Then $S \cup D_1$ is isotopic to $S\cup D_2$.
\hip
\noindent (iv)  Suppose $S$ has two complementary annuli $A_1$ and $A_2$  such that $S\cup A_1\cup A_2$ is an embedded scalloped surface containing a  torus containing $A_1$ and $A_2$ as $\pi_1$-injective annuli.  Then $S \cup A_1$ is isotopic to $S\cup A_2$.

\end{lemma}

\begin{proof} (i) Let $\alpha_0$ and $\alpha_1$ be complementary smooth arcs in $\bdry D$, with outward cusps at $\bdry \alpha_0$.  Now, assuming we have chosen a metric on $M$, we can isotope $S$ and $D$ in $M$ such that $\alpha_0$ is near an arc in $\bdry S\cup D$, as shown in Figure \ref{DigonFig}.   (The metric is only for purposes of visualization.)  Then we choose a product $Q=\beta\times[0,1]$ with $\beta\times\{0\}\subset \bdry (S\cup D)$ and $\beta\times \{1\}\subset \alpha_0$.    To construct an isotopy from $S\cup D$ to $S\cup E$ in $M$, we move $\beta_0\subset \bdry S$ to $\alpha_1$ through the product $Q\cup D$ in $S$, extending the isotopy to  $Q\cup D$.   Then we move $\alpha_1\subset \bdry E$ through the product $Q\cup E$ back to $\alpha_0$, extending the isotopy and returning $Q$ to its original position.   Extending the isotopy to all of $S\cup D$, we obtain the desired isotopy.

To prove the more general statement, we apply the first statement repeatedly.  

(ii) The proof of the second statement is similar.  First we consider the case $n=1$.  In this case there is just one complementary annulus $A=A_1$ with $\bdry A$ attached to annulus components $S_1$ and $S_2$ of $S$,   We have another annulus $A'$ with the same boundary.   Then $S_1\cup A \cup S_2$ is isotopic to $S_1$ and $S_1\cup A' \cup S_2$ is isotopic to $S_1$.

(iii)  Suppose $F$ is the component of $S\cup D_1\cup D_2$ containing the two digons.   If we remove the interior of $D_1$, we obtain a surface $F'$ say.   Now choose an isotopy of $D_1\cup D_2$ in $F$ which exchanges the digons.    Extending the isotopy, we have an isotopy of $F$ in itself which exchanges the digons.   This shows that, up to isotopy, removing the interior of $D_1$ yields the same scalloped surface as removing $D_2$.  

(iv) The proof of this part is the similar to (iii), but now we need the annuli $A_1$ and $A_2$ to be essential and disjoint in the torus component $F$ of $S\cup A_1\cup A_2$ which contains $A_1$ and $A_2$ to ensure that there is an isotopy of $A_1\cup A_2$ in $F$ which exchanges the annuli $A_1$ and $A_2$.  
\end{proof}

\begin{figure}[ht]
\centering
\scalebox{0.4}{\includegraphics{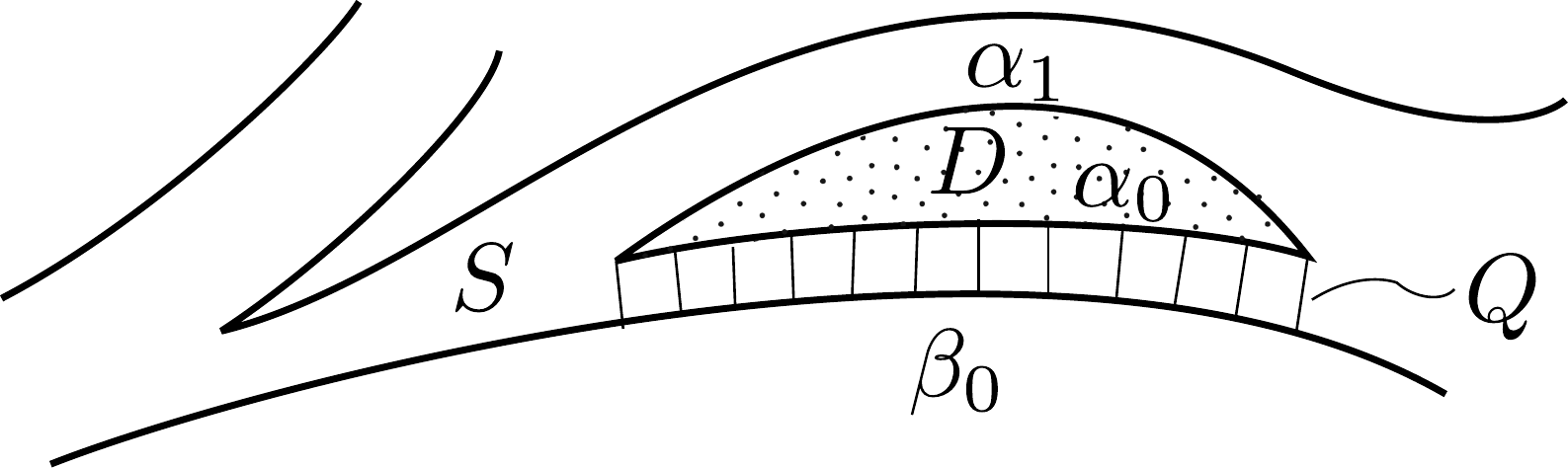}}
\caption{\footnotesize  Digon for a scalloped surface $S$.}
\label{DigonFig}
\end{figure}

\begin{proof}[Proof of Proposition \ref{AdigonalProp}]   We begin by explaining how to obtain the canonical underlying scalloped surface associated to a given lamination link represented by $V_\bou(\rho)$.   Later, we shall prove that this scalloped surface is unique up to isotopy.   Ignoring the fibered structure of $V_\bou(\rho)$ we obtain a scalloped surface with an oriented measured foliation, $(S_0,\F_0,\nu_0)$.   We begin by splitting on every reducing arc to ensure that $\F_0$ is irreducible.   In particular, this means we obtain a new $S_0$ such that any component of $S_0$ with a closed smooth boundary component is an annulus.  Now we attach a maximal disjoint collection of complementary digons and  annuli and pinch along these digons annuli to obtain $(S_1,\F_1,\nu_1)$.   However, we choose a maximal collection subject to the condition that no component of $S_1$ is closed, or in other words subject to the condition that no component of $S_1$ is a torus.  This means, by Lemma \ref{ReducingLemma}, that the only components of $S_1$ which have complementary digons are embedded tori with the interior of a digon removed; the only components of $S_1$ which have complementary annuli are annuli.  

(a)  Now we show that if $(S,\F,\mu)$ is a scalloped surface with an irreducible, adigonal, and anannular measured oriented foliation, then this can be represented by a prelamination $V_\bow(\tau)$.  Pinching on any complementary digons or annuli for $S$, we obtain $(S',\F',\mu')$.   Now $S'$ may contain torus components with irrational slope foliations or foliations by closed curves.   

From  $(S',\F',\mu')$ we can recover a lamination link represented in the form $V_\bow(\tau)$, with underlying scalloped surface $S$ as follows.   We first consider a torus component $F$ of $S'$.   If $F$ is foliated by circles, we cut on a single circle leaf to obtain an annulus foliated by circles.   It is easy to choose a vertical foliation of $F$, transverse to $\F'$ which gives $F$ the structures of a train track neighborhood of a closed curve $\theta$.  In this way we represent the foliation of $F$ by $V_\boy(\theta)$. Next suppose $F$ is a torus component of $S'$ with an irrational slope foliation.  We choose a ``vertical" foliation of $F$ by circles which is transverse to $\F'$.  By splitting on sufficiently long arc in a (dense) leaf, the leaves of the vertical lamination are cut into arcs, which means that we get  a prelamination $V_\boy(\theta)$, say, where $V(\theta)$ has a single complementary digon.    If $F$ is a component of the scalloped surface $S'$ with at least two inward cusps on each boundary component, then by Lemma \ref{ReducingLemma}, the union of separatrices of $\F'$ in $F$ is dense.  We choose any ``vertical" foliation $\V$ of $F$ transverse to $\F'$ in $F$.  If we regular-split on sufficiently long arcs with one end at a cusp, we ensure that all leaves of the vertical foliation are cut into compact segments, and we obtain a representation $V_\boy(\theta)$ of this component of the link.  Combining the various prelaminations $V_\boy(\theta)$ obtained from different components of $S'$, we obtain a prelamination $V_\bow(\tau)$ representation of  $(S,\F,\mu)$.

(b)  We now know that a lamination link can equally well be represented as an irreducible, adigonal, anannular measured oriented foliation of a  of  scalloped surface or as an irreducible, adigonal, anannular prelamination.   Our next goal is to show that the underlying scalloped surface associated to a lamination link is unique up to isotopy.

Suppose $V_\bou(\rho)$ and $V_\bow(\tau)$ are equivalent representatives of a lamination link in $M$,  and both are irreducible, adigonal, and anannular..   We will show that they have isotopic underlying scalloped surfaces.  The fact that they are equivalent means there is a 1-parameter family of prelaminations $P_t$, $0\le t\le 1$, where $P_0=V_\bou(\rho)$ and $P_1=V_\bow(\tau)$.  Ignoring the fibered neighborhood structure, we have a 1-parameter family of scalloped surfaces with foliations, say $Q_t=(S_t,\F_t,\mu_t)$.   For the 1-parameter family $Q_t$, there are finitely many ``events" each of which corresponds to:  (i) splitting on a reducing arc if $V_\bou(\rho)$ is reducible or pinching to introduce a reducing arc;  (ii) pinching on a complementary digon or splitting on an arc to introduce a complementary digon; and (iii) pinching on a complementary annulus or cutting on a closed leaf to introduce a complementary annulus.  Regular splitting does not change the isotopy class of $S_t$ or $Q_t$.   Suppose now that on an interval of time there is a reducing arc in $Q_t$, then we split to eliminate the reducing arc throughout the interval of time.  We do this for all 1-parameter families of reducing arcs.  This ensures that $Q_t$ is irreducible at all times.   Similary, if $Q_t$ is not adigonal on an interval of time, and a certain digon $D_t$ persists in that interval of time, then we pinch throughout the interval of time, choosing a maximal time interval in which the complementary digon exists.   By Lemma \ref{DigonPinch}, this uniquely determines the isotopy class of $S_t\cup D_t$.  In this way we construct a 1-parameter family $Q_t$ such that the isotopy class of $S_t$ is unchanged over the interval $[0,1]$.  There is one situation where we avoid pinching on a digon over a ``maximal" time interval.  Namely, if $S_t$ has a component $F_t$ which is a torus with several digons excised over a time interval, then one must pinch on all but one digon.    Since there is a choice as to which complementary digon to leave, we must allow two complementary digons to coexist over some very short subintervals of $[0,1]$.   But Lemma \ref{DigonPinch}(iii) shows that pinching one of them gives the same scalloped surface, up to isotopy, as pinching the other.   Thus even in this situation, the isotopy class of the scalloped surface $S_t$ is uniquely determined if we agree that over intervals where two digons coexist, we let $S_t$ denote the scalloped surface obtained by pinching either one of them.  

If $P_t$ fails to be anannular over an interval of time, we similarly pinch throughout a maximal time interval during which some complementary annulus persists.  If we do this for as many complementary annuli as possible, without introducing torus components of $S_t$, we can arrange that $P_t$ is anannular at all times.   Once again, there is a caveat.  Pinching different complementary annuli  in different time intervals may make it necessary to allow two complementary annuli to coexist over some very short time intervals.   Once again, letting $S_t$ denote the isotopy class of the scalloped surface obtained by pinching on either of these annuli, then using Lemma \ref{DigonPinch}(iv), we show $S_t$ is well-defined up to isotopy.

After all of these modifications, $S_t$ yields an isotopy of scalloped surfaces.
 \end{proof}

Next we will show that the supremum in the definition of $X(V_\bow(\tau))$ can be achieved by a Seifert lamination.

\begin{proof}[Proof of Theorem \ref{Extend}] We represent a lamination link $(L,\mu)$ in a compact, orientable 3-manifold as $V_\bow(\tau)\embed \intr(M)$.  Suppose $V_\bow(\tau)$ is a representative with the canonical (up to isotopy) underlying scalloped surface; this means that it is irreducible, adigonal, and anannular.   We can then analyze Seifert laminations for $V_\bow(\tau)$  viewed as a peripheral link in the knot exterior $\hat M$.  If  $V_\bow(\tau)$ is viewed as a peripheral lamination in $\hat M$, we have
$$\hat X(V_\bow(\tau))=\sup\left\{\Chi(B(\bov)):(B,\bdry B) \embed (\hat M,\bdry \hat M),\ \bdry B(\bov)=\tau(\bow),\ B \text{ aspherical}\right\}.$$

\noindent When $V_\bow(\tau)$ is viewed as a peripheral lamination in $\hat M$, we know that $\hat X(V_\bow(\tau))$ can be achieved as $\Chi(B'(\bov'))$, where $\bdry B'(\bov')=V_\bow(\tau)$, and $B'$ is properly embedded in $\hat M$.
We want to show that a Seifert lamination in $\hat M$ which achieves this supremum also achieves the supremum among all Seifert laminations in $M$, namely, that it achieves the supremum 
$$X(V_\bow(\tau))=$$
$$\sup\left\{\Chi(B(\bov)):B \embed M,\ \bdry B(\bov)=V_\boy(\theta),  
V_\boy(\theta)\text{ equivalent to } V_\bow(\tau),  B \text{ aspherical}\right\}.$$

So we suppose $V_\bou(\rho)\embed M$ is an arbitrary representative for the lamination link $V_\bow(\tau)$.   This representative may not be irreducible, and may have complementary digons and annuli.   We also suppose that there is a ``better" Seifert lamination $B_0(\bov_0)$ with $\bdry B_0(\bov_0)=V_\bou(\rho)$.  For the arbitrary representative $V_\bou(\rho)$, we can construct a submanifold $\hat M_0$ such that $V_\bou(\rho)$ is peripheral in $\hat M_0$ just as before:  remove a product regular neighborhood $N=V(\rho)\times [0,1]$ from the interior of $M$, and let $V(\rho)\subset\bdry \hat M_0$.  Then $B_0(\bov_0)$ can be properly embedded in $\hat M_0$.  We will modify $V_\bou(\rho)\subset \bdry \hat M_0$  to obtain $V_\bow(\tau)\subset \bdry\hat M_1$.   If $B_0(\bov_0)$ is a maximal-$\Chi$ (or taut) Seifert lamination for $V_\bou(\rho)$ in $ \hat M_0$, we will show that as we modify $V_\bou(\rho)$ we can also modify $B_0(\bov_0)$ without decreasing $\Chi(B_0(\bov_0))$.

To simplify notation, let $P_0$ denote $V_\bou(\rho)\embed M$ and let $P_1$ denote $V_\bow(\tau)$.   We assume these are equivalent, so there is a 1-parameter family $P_t$ of prelaminations, $0\le t\le 1$ with isolated events changing the isotopy class of $P_t$ or its underlying scalloped surface $S_t$.   The isolated events correspond to the following operations:  (i) split on reducing arcs if $P_t$ is reducible or pinch to introduce a reducing arc;  (ii) pinch on a complementary digon or split on an arc to introduce a complementary digon; and (iii) pinch on a complementary annulus or cut on a closed leaf.   Regular splitting or pinching may cause events changing the combinatorial type of the train track $\rho_t$ corresponding to $P_t$, but it does not change the underlying scalloped surface $S_t$.  We are only interested in the events which change $S_t$.  For each $P_t$ we construct $\hat M_t$ so that $P_t$ is peripheral in $\hat M_t$ as before.  Further, we denote by $\hat X_t$ the $X$ function on $\CB(\rho_t)$ when $\rho_t$ is peripheral in $\hat M_t$.    We also simplify notation by writing $\Upsilon_0$ for $B_0(\bov_0)$.   We shall inductively construct  $\Upsilon_t=B_t(\bov_t)$  for $0<t\le 1$  properly embedded in $\hat M_t$ such that $\bdry \Upsilon_t=P_t$ , and $\Upsilon_t$ is taut in $\hat M_t$ ($B_t$ is taut in $\hat M_t$).   The goal is to show that we can find $\Upsilon_t$ so that $\Chi(\Upsilon_t)$ does not decrease, to conclude that $\Upsilon_0$ is not ``better" than $\Upsilon_1$.  

We begin by modifying the family $P_t$ globally.   Whenever there is a reducing arc in $P_t$ which persists in a time interval we cut $P_t$ throughout the interval, and we also modify $\hat M_t$ such that $\hat M_t$ is obtained by removing from $M$ the interior of a product of the form $S_t\times I$.    Whenever we split $P_0$ in this process, we also split $\Upsilon_0$ in a small neighborhood of $\bdry \hat M_0$.    Repeating this as often as possible, we can now assume that $P_t$ is irreducible for all $t$.   Notice that if there is an event which splits on an arc in a closed leaf, this introduces a reducing arc.   Hence, after our modification, the event becomes more radical: now the only event splitting on a portion of a closed leaf instantly splits on the entire closed leaf.

Suppose we have already constructed $\Upsilon_t$ for $t\in [0,r]$ such that it is carried by a taut branched surface $B_t$ and such that $\Chi(\Upsilon_t)$ is a non-decreasing funtion of $t$.  If there is a single event in the interval $(r,s)$, and no event at $t=r$ or $t=s$, we show how to define $\Upsilon_s$ after the event. 

(1)  Suppose the event splits on an arc (or closed curve) of a leaf of $P_r$ to introduce a complementary digon (complementary annulus).  In the case that the event splits on an arc in a leaf, we know immediately that the arc does not lie on a closed leaf.   In any case, we can modify $\Upsilon_r$ to get $\Upsilon_s$ simply by splitting $\Upsilon_r$ in a neighborhood of the splitting arc in $P_r$.   The event, of course, also changes $\hat M_t$ at some time in $(r,s)$.  It does not change the Euler characteristic.  If $\Upsilon_s$ is not maximal-$\Chi$, or not taut, we can easily replace it by a taut $\Upsilon_s$, without decreasing the Euler characteristic.

(2)  Suppose the event pinches a complementary digon $D$ of $P_r$ to yield $P_s$, with a corresponding modification of $\hat M_r$.    This case is more interesting, because $\Upsilon_r$ could intersect the digon $D$.   Before pinching on the digon, one must modify $\Upsilon_r$.     Actually, we will modify a weighted surface which approximates $\Upsilon_r$.  We can choose a weighted surface $\delta F_r$, $\delta>0$, carried by the taut $B_r$ to approximate $\Upsilon_r$ as closely as we wish.   Of course $\delta \bdry F_r$ also approximates $P_r$ and $\bdry F_r$ is carried by the train track $\rho_r$ corresponding to $P_r$.  The surface $F_r$ is taut, therefore incompressible.   We suppose $\bdry F_r$ is transverse to fibers in $N(\rho_r)$ and the digon is an embedded rectangular disk $E$ with two opposite sides in $\bdry_vN(\rho_r)$ and the other two sides in $\bdry_hN(\tau)$.  We can give $E$ a product structure with ``vertical" fibers including the sides in $\bdry_vN(\rho_r)$.  We isotope $F_r$ so it is transverse to $E$.   Then all curves of intersection bound disks in $E$.   A curve $\alpha$ of intersection innermost in $E$ bounds a disk $D$ in $E$.   Since $F_r$ is incompressible, $\alpha$ also bounds a disk $D'$ in $F_r$.  If we surger $F_r$ on $D$, we obtain a new surface which intersects $E$ in fewer curves.  The surgery produces a sphere, which we discard, and $\Chi(\delta F_r)$, is unchanged.  Repeating, we modify $F_r$ until there are no curves of intersection with $E$, and the combined modifications leave $\Chi(\delta F_r)$ unchanged.  Now we can pinch on the digon, with its fibered structure to obtain $N(\rho_s)=N(\rho_r)\union E$.  The modified $F_r$ becomes a surface $F_s$ with $\bdry F_s=\bdry F_r$ carried by $V(\rho_s)$.  We have $\Chi(\delta F_s)\ge \Chi(\delta F_r)$.    By the continuity of $\hat X_r$  on $\CB(\tau_r)$ and the continuity of $\hat X_s$ on $\CB(\tau_s)$ we conclude $\hat X_s(P_s)\ge \hat X_r(P_r)$.  By by Proposition \ref{SumProp}, we can then find a taut $\Upsilon_s$ such that $\Chi(\Upsilon_s)=X(P_s)$.   

\begin{figure}[ht]
\centering
\scalebox{0.75}{\includegraphics{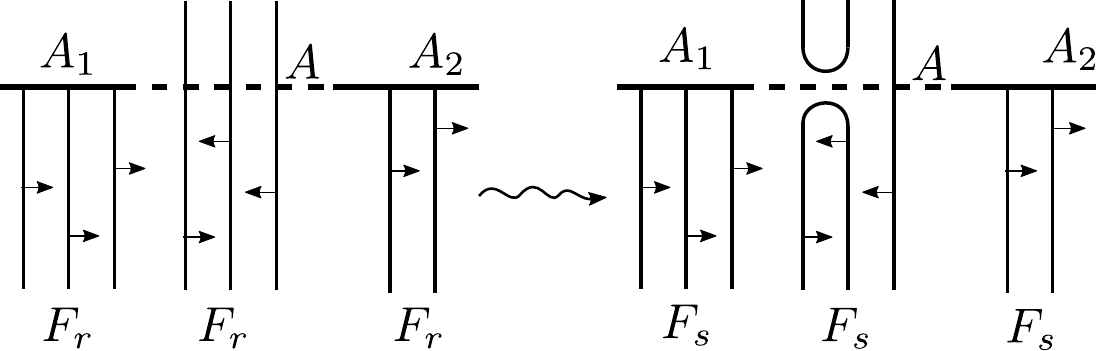}}
\caption{\footnotesize  Opposing transverse orientations in $A$.}
\label{Annulus1}
\end{figure}

{\begin{figure}[H]
\centering
\scalebox{0.75}{\includegraphics{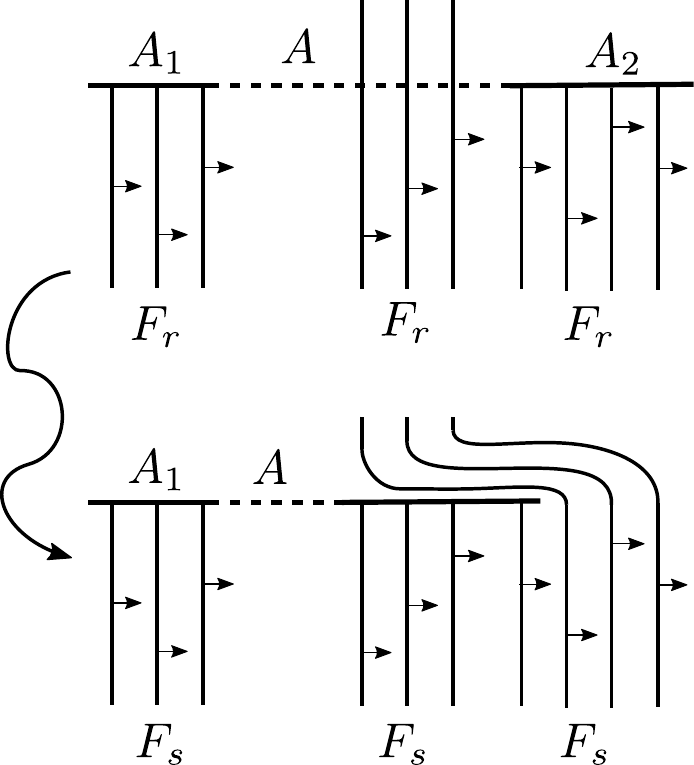}}
\caption{\footnotesize  Consistent transverse orientations in $A$, $A_1$, $A_2$.}
\label{Annulus2}
\end{figure}

\begin{figure}[H]
\centering
\scalebox{0.75}{\includegraphics{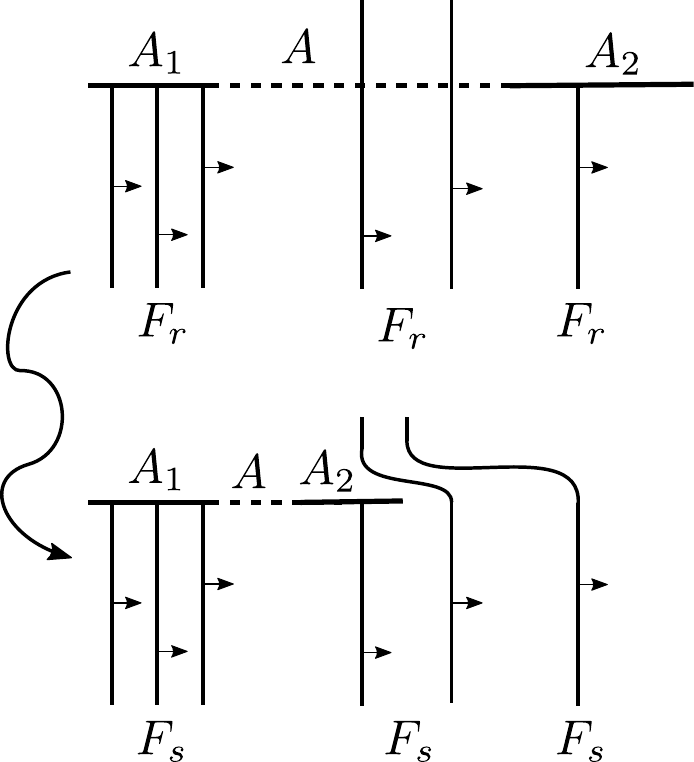}}
\caption{\footnotesize   Consistent transverse orientations in $A$, $A_1$, $A_2$.}
\label{Annulus3}
\end{figure}}

(3)  Suppose the event pinches a complementary annulus $A$ of $P_r$ to yield $P_s$.    Because $P_r$ is an irreducible prelamination, the components of $P_r$ adjacent to $A$ must be foliated annuli $A_1$ and $A_2$, where $A_1$ and $A_2$ are regarded as fibered neighborhoods of closed curve train tracks.  Furthermore, the orientations of leaves in $A_1$ and $A_2$ must be consistent, such that pinching on $A$ gives a new oriented measured foliation on $P_s$. Once again, if $\Upsilon_r$ in $\hat M_r$ is a taut Seifert lamination for$P_r$, carried by a taut branched surface $B_r$, we approximate by a weighted taut surface $\delta F_r$ carried by $B_r$.  The surface $ F_r$ includes boundary curves transverse to fibers in $A_0$ and $A_1$, which are oriented and transversely oriented.  We consider transverse intersections of $F_0$ with $A$.    As in (2), we can eliminate closed curves of $F_0\cap A$ which are inessential in $A$, and obtain a new $F_0$ with only curves of intersection essential in $A$.   The new $\Chi(\delta F_0)$ is unchanged.  Since $F_r$ has a transverse orientation in $\hat M_r$, the curves of intersection also have transverse orientations in $A$.   If not all the curves of intersection have consistent transverse orientations, we find two adjacent ones with opposite orientations and perform an oriented cut-and-paste to ``cancel" the curves of intersection, see Figure \ref{Annulus1}.   (Figures \ref{Annulus1},\ref{Annulus2},\ref{Annulus3},\ref{Annulus4},\ref{Annulus5} become more realistic if they are rotated about a vertical axis on the left or right of the figure.)  This does not change $\Chi(F_r)$ either.   Repeating, we obtain a new $F_r$ so that all (essential) curves of intersection with $A$ have consistent transverse orientations.   There are two cases to consider:  Either the curves of intersection have transverse orientations consistent with those of the curves of $\bdry F_1$ in $A_2$ and $A_1$, or they have orientations that are inconsistent.   In case all the curves in $A$, $A_1$, and $A_2$ have consistent transverse orientations, Figures \ref{Annulus2} and \ref{Annulus3} show schematically how to modify $F_r$ without changing $\Chi$.    The final result of all modifications is $\delta F_s$.   In case all the curves in $A$ have transverse orientations inconsistent with those in $A_1$ and $A_2$ Figures \ref{Annulus4} and \ref{Annulus5} show schematically how to modify $F_r$ without changing  $\Chi$ to construct $F_s$.  
In case the transverse orientations of curves in $A$ are inconsistent with those in both $A_1$ and $A_2$, Figures \ref{Annulus4} and \ref{Annulus5} show schematically how to modify $F_0$ without changing $\Chi(F_0)$.   In both cases, the continuity of $\hat X_r$ and $\hat X_s$ imply that we can choose $\Upsilon_s$ with $\Chi(\Upsilon_s)=\Chi(\Upsilon_r)$ by Proposition \ref{SumProp}.

\begin{figure}[ht]
\centering
\scalebox{0.75}{\includegraphics{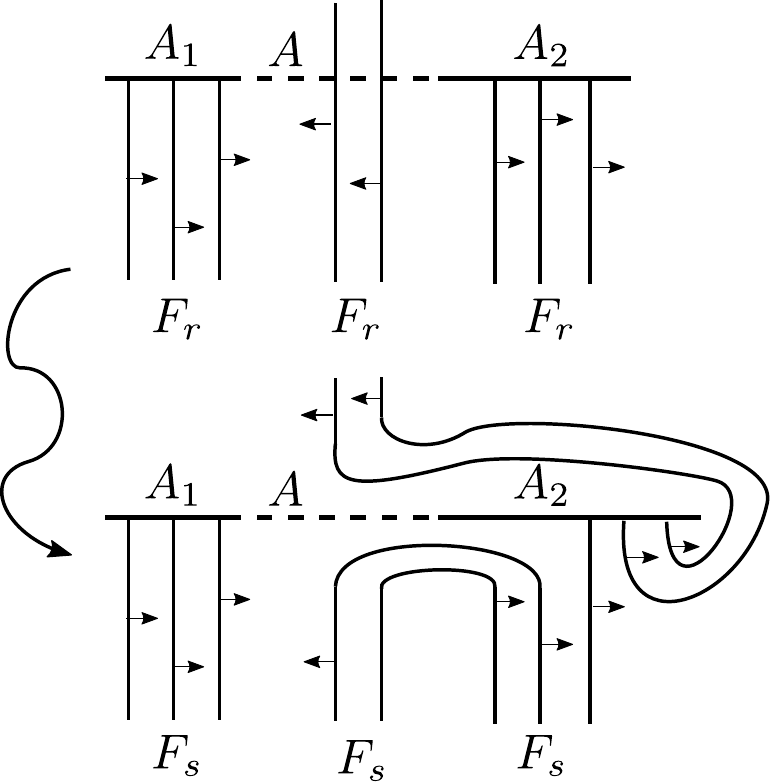}}
\caption{\footnotesize  Opposite transverse orientations in $A$ and $A_2$.}
\label{Annulus4}
\end{figure}

\begin{figure}[H]
\centering
\scalebox{0.75}{\includegraphics{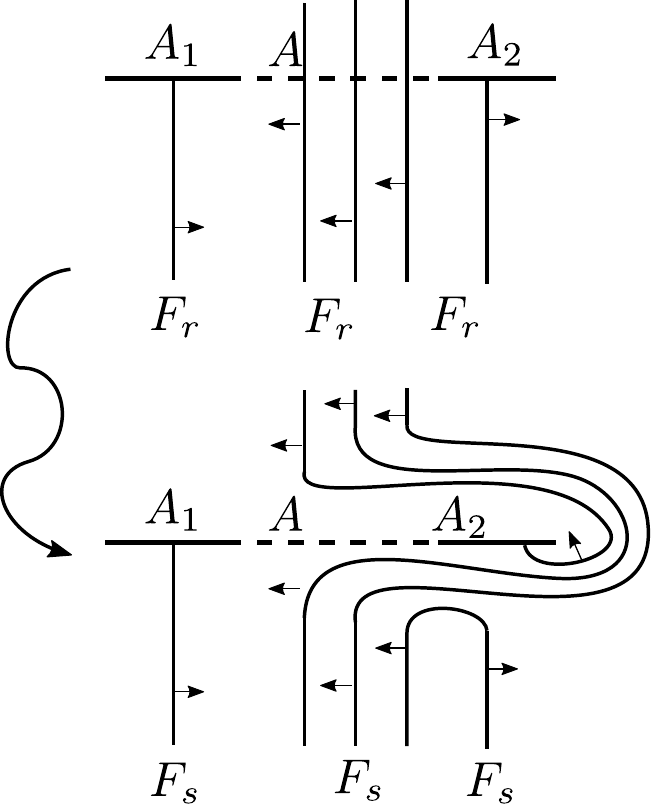}}
\caption{\footnotesize  Opposite transverse orientations in $A$ and $A_2$.}
\label{Annulus5}
\end{figure}

After finitely many induction steps, we have constructed Seifert laminations $\Upsilon_t$ for all $t\in [0,1]$ (except at times $t$ corresponding to events).   We have shown $\Chi(\Upsilon_t)$ is non-decreasing.   This shows that $\hat X_0(P_0)\le \hat X_1(P_1)$.   We cannot find a better Seifert lamination by choosing a different prelamination representative of the lamination link $V_\bow(\tau)$.
\end{proof}

\begin{proof}[Proof of Theorem \ref{Continuity}]  Theorem \ref{Extend} says that that a given a lamination link represented as $V_\bow(\tau)$ which is irreducible, adigonal anannular, there is a maximal-$\Chi$ Seifert lamination in the track exterior $\hat M$.  The link $V_\bow(\tau)\subset \hat M$ is peripheral in $\hat M$.   Now we take a slightly different point of view.   Given a framing $V(\tau)$ of a train track $\tau\embed M$, we consider all the invariant weight vectors $\bow$ such that $V_\bow(\tau)$ is irreducible, and can therefore be regarded as peripheral in the same $\hat M$.    This is the subspace $\CIB(V(\tau))$ of $\C(\tau)$.   Since $V(\tau)\embed \bdry \hat M$ and $\bow\in \CIB(\tau)$, a Seifert lamination for $V_\bow(\tau)$ is maximal-$\Chi$ or taut in $\hat M$ if and only if it is taut in $M$.   Now by Proposition \ref{SumProp}, there exist finitely many $\tau$-taut branched surfaces.  $\{T_1,T_2,\ldots T_p\}$ of $\tau$-taut branched surfaces in $\hat M$ such that for every $\bow\in\CB(\tau)$,  there is a taut $T_i(\bov)$ such that $\bdry T_i(\bov)=\tau(\bow)$.   This shows that $X$ is continuous (and piecewise linear) on $\CIB(\tau)$, where $X$ refers to the function defined for Seifert laminations in $M$.  \end{proof}

\begin{proof}[Proof of Theorem \ref{Trivial}]  Suppose first $V_\bow(\tau)$ is peripheral in $M$.  A trivial leaf $\alpha$ of the lamination link, in this case, is any leaf which is null-homotopic in $M$.   If $B(\bov)$ represents a taut Seifert lamination for $\tau(\bow)$, its leaves are $\pi_1$-injective.   This implies that the leaf $\ell$ of $B(\bov)$ containing $\alpha$ must be a disk.   

Now suppose $V_\bow(\tau)$ is not peripheral.  In this case a trivial leaf is a leaf which bounds a disk whose interior is disjoint from the lamination link.  Suppose $V_\bow(\tau)$ is irreducible, adigonal, and anannular.  This means that trivial leaves lie in annulus components of  $V(\tau)$ corresponding to closed curves of $\tau$.  Then there is a taut Seifert lamination $B(\bov)$ in the track exterior $\hat M$.   Now we are dealing with a peripheral lamination link in $\hat M$, so the above argument shows every trivial leaf bounds a disk in $\hat M$.
\end{proof}

\bibliographystyle{amsplain}
\bibliography{ReferencesUO3}
\end{document}